\documentclass[a4paper,english,12pt]{amsart}
\usepackage[margin=3cm]{geometry} 

\usepackage[latin1]{inputenc}
\usepackage[bottom]{footmisc}
\usepackage{mathrsfs}
\usepackage{amssymb}
\usepackage{amsmath} 
\usepackage{amsthm}
\usepackage{commath}
\usepackage{mathtools}
\usepackage{amsfonts}
\usepackage{tikz}
\usepackage{bm}
\usepackage{accents}
\usepackage{tikz-cd}
\usetikzlibrary{matrix}
\usepackage{stmaryrd}
\usepackage{hyperref}
\usepackage[shortlabels]{enumitem}

\newcommand{\QQ}{\mathbb{Q}}

\newcommand{\GG}{ G}
\newcommand{\RR}{\mathbb{R}}

\newcommand{\CC}{\mathbb{C}}

\newcommand{\oF}{\overline F}
\newcommand{\ZZ}{\mathbb{Z}}

\newcommand{\doots}{,\dots ,}

\newcommand{\Akk}{\mathfrak A_4}

\newcommand{\OO}{\mathcal O}

\newcommand{\Fv}{F_v}

\newcommand{\sG}{\prescript{}{\sigma}G}

\newcommand{\Ov}{\mathcal O _v}

\newcommand{\vMF}{v\in M_F}
\newcommand{\vMFz}{v\in M_F ^0}
\newcommand{\vMFi}{v\in M_F ^\infty}

\newcommand{\AAA}{\mathbb A}

\newcommand{\yyy}{\mathbf{y}}

\newcommand{\sJ}{\prescript{}{\sigma}J}

\newcommand{\wH}{\widehat H}
\newcommand{\wchi}{\widetilde{\chi}}

\newcommand{\piv}{\pi_v}

\newcommand{\BGA}{BG(\AAA)}
\newcommand{\Gfs}{G(\overline{F})}
\newcommand{\muu}{\boldsymbol{\mu}}

\DeclareMathOperator{\Aut}{Aut}

\DeclareMathOperator{\coun}{count}
\DeclareMathOperator{\Conj}{Conj}

\DeclareMathOperator{\ind}{ind}
\DeclareMathOperator{\End}{End}

\DeclareMathOperator{\ord}{ord}

\DeclareMathOperator{\Pic}{Pic}

\DeclareMathOperator{\coker}{coker}
\DeclareMathOperator{\Hom}{Hom}
\DeclareMathOperator{\Gal}{Gal}
\DeclareMathOperator{\GL}{GL}
\DeclareMathOperator{\SL}{SL}
\DeclareMathOperator{\cycl}{cycl}

\DeclareMathOperator{\Imm}{Im}

\DeclareMathOperator{\supp}{supp}

\DeclareMathOperator{\Spec}{Spec}

\DeclareMathOperator{\cont}{cont}

\DeclareMathOperator{\charr}{char}
\DeclareMathOperator{\Tr}{tr}
\DeclareMathOperator{\tr}{tr}

\DeclareMathOperator{\Frob}{Frob}
\DeclareMathOperator{\un}{un}

\DeclareMathOperator{\inv}{inv}
\DeclareMathOperator{\tame}{{tame}}

\DeclareMathOperator\et{\acute et}
\def\no{n\textsuperscript{0}\,}

\DeclareMathOperator{\Br}{Br}

\newtheorem{mydef}[equation]{Definition}
\newtheorem{lem}[equation]{Lemma}
\newtheorem{thm}[equation]{Theorem}

\newtheorem{conj}[equation]{Conjecture}

\newtheorem{prop}[equation]{Proposition}

\newtheorem{rem}[equation]{Remark}

\newtheorem{cor}[equation]{Corollary}

\newtheorem*{theorem*}{Theorem}

\numberwithin{equation}{subsection}
\usepackage{blindtext}

  \DeclareFontFamily{U}{wncy}{}
    \DeclareFontShape{U}{wncy}{m}{n}{<->wncyr10}{}
    \DeclareSymbolFont{mcy}{U}{wncy}{m}{n}
    \DeclareMathSymbol{\Sh}{\mathord}{mcy}{"58} 
\keywords{Malle conjecture, $G$-torsor, Poisson formula, Generalized discriminant}
\title{Torsors for finite group schemes of bounded height}
\author[Darda]{Ratko Darda}
\address{Department of Mathematics and Computer Science\\ University of Basel}
\email{ratko.darda@gmail.com}
\author[Yasuda]{Takehiko Yasuda}
\address{Department of Mathematics\\ Graduate School of Sciences\\Osaka University}
\email{yasuda.takehiko.sci@osaka-u.ac.jp}
\begin{document}
\begin{abstract}
Let~$F$ be a global field. Let~$G$ be a non trivial finite \' etale tame $F$-group scheme. We define height functions on the set of $G$-torsors over~$F,$ which generalize the usual heights such as discriminant. As an analogue of the Malle conjecture for group schemes, we formulate a conjecture on the asymptotic behaviour of the number of $G$-torsors over~$F$ of bounded height. This is a special case of our more general ``Stacky Batyrev-Manin" conjecture \cite{dardayasudabm}. The conjectured asymptotic is proven for the case when~$G$ is commutative. When~$F$ is a number field, the leading constant is expressed as a product of certain arithmetic invariants of~$G$ and a volume of a space attached to~$G$. Moreover, an equidistribution property of $G$-torsors in the space is established. 
\end{abstract}
\maketitle
\section{Introduction}
\subsection{Malle conjecture} A classical result, the Hermite-Minkowski theorem, states that the number of extensions of a number field~$F$ having bounded discriminant is finite. It is a very natural question to ask what is this number. Moreover, one can fix  an invariant, such as the Galois group of the Galois closure, and count such extensions. In this context, we have the famous Malle conjecture:
\begin{conj}[{Malle \cite{Malle}}] \label{classmalle}
Let~$G$ be a non-trivial group which is embedded as a transitive subgroup of the group of permutations~$\mathfrak S_n$ for some $n\geq 1$. For an extension $K/F$ we write $N(\Delta(K/F))$ for the norm of its discriminant and $\Gal(K/F)=G$ if the Galois group of the Galois closure of~$K/F$ is isomorphic to~$G$ as a permutation group. There exists~$C>0$ such that
\begin{equation*}\#\{K/F|[K:F]=n, \Gal(K/F)=G, N(\Delta(K/F))<B\} \sim_{B\to\infty} C B^a\log(B)^{b-1},\end{equation*}
for some explicit invariants~$a=a(G\leq \mathfrak S_n)$ and~$b=b(G\leq \mathfrak S_n, F)$.
\end{conj}
\subsubsection{}The conjecture can be stated for global fields. We list a few of the cases when the Malle conjecture is known to be true:
\begin{itemize}
\item~$G$ is abelian, of order coprime to the characteristic of~$F$ if the characteristic is positive, embedded in its regular representation~\cite{Wright};
\item~$F$ is not of characteristic~$2$, and~$G=\mathfrak S_i$ for $i\in\{3,4,5\}$ in its natural representation \cite{https://doi.org/10.48550/arxiv.1512.03035}, building upon \cite{dave}, \cite{densityquartic}, \cite{densityquintic};
\item~$F$ is a number field, $G$ is~$\mathfrak S_3$ embedded in its regular representation~ \cite{sexticsthree};
\item~$F$ is a number field and $G=\mathfrak S_i\times A$ for $i\in\{3, 4, 5\}$, where~$\mathfrak S_i$ is in its natural representation and~$A$ is abelian in its regular representation \cite{https://doi.org/10.48550/arxiv.2004.04651}, \cite{wangmalle}.
\end{itemize}
A counterexample to Conjecture \ref{classmalle} has been given in \cite{Kluners} by Kl\" uners for $G=(\ZZ/3\ZZ)\wr (\ZZ/2\ZZ)$ embedded in its degree~$6$ representation. 
It rests unproven for some usual groups such as the alternating group~$\Akk$ embedded in its standard or regular representation. 
\subsubsection{}\label{malleortorsors} Let~$G_0$ be a point stabilizer for a permutation group~$G\hookrightarrow~\mathfrak S_n.$ For a {\it connected} $G$-torsor~$\Spec(L),$ one has that~$L^{G_0}$ is a degree~$n$ extension of~$F$ having~$G$ for the Galois group of its Galois closure. The fibers of the map $\Spec(L)\mapsto L^{G_0}$ have all equal finite cardinality (by e.g. \cite[Lemma 5.8]{https://doi.org/10.48550/arxiv.1408.3912}). 
We deduce that the Malle conjecture can be studied as a question of counting connected $G$-torsors~$\Spec(L)$, with respect to the discriminant of ~$L^{G_0}$. 
\subsection{Non-constant $G$} In this article, we will study the question of counting $G$-torsors for non-constant finite group schemes~$G$. This question appears natural on its own, but let us provide some more motivation to study it. 
\subsubsection{} Our first motivation comes from recent attempts \cite{darda:tel-03682761}, \cite{Ellenberg} to extend a conjecture of Manin on number of rational points of bounded height on {\it Fano} varieties \cite{frankemanintschinkel} to {\it algebraic stacks}. 
The articles try to explain {\it \` a la Manin} some counting results of points of stacks (such as~\cite{Hortsch}), as well as similarities between the Manin and the Malle conjectures, observed by the second named author in \cite{Yasuda}. For that purpose, definitions of heights for stacks are provided. The generalization of the Manin conjecture for stacks has not been completed: in \cite{darda:tel-03682761} one is restricted to {\it weighted projective stacks}, while in \cite{Ellenberg}, the authors predict the exponent of~$B$ up to an~$\epsilon$ in an asymptotic formula on the number of points, but no leading constant or the exponent of the~$\log(B)$. In the article \cite{dardayasudabm} we will work on this question.

For a finite group scheme~$G$, we have the {\it classifying stack}~$BG$. It is of dimension zero and its set of rational points~$BG(F)$ is precisely the set of $G$-torsors. It looks like a very natural candidate to try to formulate more precise conjecture for the numbers of rational points of bounded height (a ``Manin conjecture"). If~$G$ is constant, the Malle conjecture predicts the number of rational points of~$BG$ for the height given by the discriminant (or the discriminant of a fixed field as above for some embedding $G\hookrightarrow \mathfrak S_n$). But it does not tell anything for non-constant~$G$ and hence does not explain counting results such as \cite[Corollary 9.2.5.10]{darda:tel-03682761} when~$G=\muu_m$ is the finite group scheme of $m$-th roots of unity. 
\subsubsection{} Our second motivation is that torsors of non-constant group schemes may influence the counting of the torsors of constant groups.  Let~$\sG$ be an inner twist (see Paragraph \ref{paragtwist}) of a finite group~$G$. It is a finite group scheme, but usually non-constant. There exists a canonical bijection from the set of~$\sG$-torsors to the set of~$G$-torsors. For a closed subgroup~$R$ of~$\sG,$ we have a map from the set of~$R$-torsors to the set of~$\sG$-torsors which has fibers of bounded cardinality. We obtain a canonical map, with fibers of bounded cardinality, from the set of~$R$-torsors to the set of~$G$-torsors. Thus the torsors of the non-constant finite group~$R$ may influence the number of~$G$-torsors. 
%
\subsection{Principal results}
Let us state our prediction for the number of $G$-torsors of bounded height and our principal results. Let~$F$ be a global field and let~$G$ be a non-trivial finite \' etale {\it tame} $F$-group scheme (we say that a finite group scheme defined over a field is tame, if the characteristic of the field is zero or if the order of the group is coprime to the characteristic). 
\subsubsection{}In Definition \ref{definitionofheight}, we define height functions $H:BG(F)\to\RR_{>0}$ which satisfy that the number of $G$-torsors of bounded height is finite (Proposition \ref{estimatenumber}). The heights capture some usual notions of ``size" of $G$-torsors (e.g. the notion of a discriminant). For heights, we have notions of its~$a$ and~$b$ invariants. 

We fix a height $H:BG(F)\to\RR_{>0}$. As we have seen above, for every inner twist~$\sG$ and every closed subgroup~$R$ of~$\sG$, we have a canonical map $BR(F)\to BG(F)$. We call such a morphism {\it breaking thin}\footnote{The terminology is borrowed from \cite{geometricconsitencymanin}.}, if in the lexicographic order, the pair consisting of~$a$ and~$b$ invariants of the pullback height (that is the composite function $BR(F)\to BG(F)\xrightarrow{H}~\RR_{>0}$) is bigger than the pair consisting of~$a$ and~$b$ invariants of~$H$. Let us say that an element of~$BG(F)$ is secure, if it is not in the image of a breaking thin map $BR(F)\to BG(F)$. As an analogue of the Malle conjecture, we propose the following conjecture.
\begin{conj}\label{introversionconj}
\begin{enumerate}
\item Suppose that~$F$ is a number field. There exists~$C>0$ such that $$\#\{x\in BG(F)|\text{$x$ is secure and }H(x)\leq B\}\sim_{B\to \infty} CB^{a(H)}\log (B)^{b(H)-1}.$$
\item Suppose that~$F$ is a function field. One has that $$\#\{x\in BG(F)|\text{$x$ is secure and }H(x)\leq B\}\asymp_{B\to \infty} B^{a(H)}\log (B)^{b(H)-1}. $$
\end{enumerate}
\end{conj}
A modification of Conjecture~\ref{classmalle} has been proposed by Turk\"eli~\cite{turkeli}. The author changes the exponent of $\log(B)$ by taking into account twists of normal subgroups~$N$ of~$G$ with~$G/N$ abelian. As part of our motivation is to extend the Manin conjecture for stacks, removal of torsors seems more natural than changing the logarithmic exponent: this what is usually done for the Manin conjecture e.g. in~\cite{peyrethin} and~\cite{geometricconsitencymanin}. In our next article \cite{dardayasudabm}, we will formulate a Manin type conjecture for certain stacks, the definition of a secure element carries naturally to this context.
\subsubsection{} We suppose that~$G$ is commutative. Then one has that every element of~$BG(F)$ is secure. The following theorem is one of our main results.
\begin{thm}\label{comcasmalle}
Conjecture \ref{introversionconj} is valid for commutative~$G$.
\end{thm}
When~$F$ is a number field and~$G$ is constant, this generalizes the above mentioned result of \cite{Wright} for heights given by discriminants and the result of Wood \cite[Theorem 3.1.]{onprobabilitieslocal} for so-called {\it fair} heights. It also generalizes the result of the first named author \cite[Theorem 8.2.2.12]{darda:tel-03682761} and \cite[Corollary 9.2.5.10]{darda:tel-03682761} for the case~$F$ is a number field and $G=~\muu_m,$ and {\it quasi-toric} and {\it quasi-discriminant} heights, respectively. We will discuss how it is compared with a result of Alberts and O'Dorney \cite{odorney} in Paragraph~\ref{detailscommutative}. 

We will define in Subsection \ref{secequid} a Radon measure~$\omega_H$ on the product space $\prod_{\vMF} BG(F_v)$ (the product goes over the set of places~$M_F$ of~$F$ and for~$v\in M_F$, the set $BG(F_v)$ is the set of $G$-torsors over~$F_v$). In Theorem \ref{satfunf} and Theorem \ref{thmofequidistribution}, we will establish an equidistribution property of $G$-torsors. Our motivation comes from the Manin conjecture, for which the analogous study of equidistribution was proposed by Peyre in \cite{Peyre}. The results are stronger than Theorem \ref{comcasmalle} and give us the following formula for the number of $G$-torsors satisfying certain local conditions.
\begin{thm}\label{equintro} Let $i:BG(F)\to\prod_{\vMF}BG(F_v)$ be the diagonal map. Let~$U\subset \prod_{\vMF}BG(F_v)$ be an open.
\begin{enumerate}  
\item Suppose that~$F$ is a number field and that the boundary of~$U$ is $\omega_H$-negligible. One has that \begin{equation*}\frac{\#\{x\in BG(F)|i(x)\in U, H(x)\leq B\}}{\# G(F)}\sim_{B\to\infty}\frac{\#\Sh^2(G)\cdot\omega_H(U)}{\# G^*(F)\cdot (b(c)-1)!}B^{a(H)}\log(B)^{b(H)-1},
\end{equation*}
where~$G^*$ denotes the Cartier dual of~$G$ and~$\Sh^2(G)$ is the second Tate-Shafarevich group of~$G$. 
\item Suppose that~$F$ is a function field. If $\omega_H(U)>0$, one has that $$\#\{x\in BG(F)|i(x)\in U, H(x)\leq B\}\asymp_{B\to\infty} B^{a(H)}\log(B)^{b(H)-1}.$$
\end{enumerate}
\end{thm}
The support of the measure~$\omega_H$ is given by the closure $\overline{i(BG(F))}\subset\prod_{\vMF}BG(F_v)$ which may be only a proper subset (when $G=\muu_m$, the conditions for the full support are given by the Grunwald-Wang theorem). In the situation of Part (1), we obtain that:
$$\lim_{B\to\infty}\frac{\#\{x\in BG(F)|i(x)\in U, H(x)\leq B\}}{\#\{x\in BG(F)|H(x)\leq B\}}=\frac{\omega_H(U)}{\omega_H(\prod_{\vMF}BG(F_v))}.$$
Hence, in particular, when~$G$ is constant and~$F$ is a number field, we have obtained an analogous formula to \cite[Theorem 2.1]{onprobabilitieslocal} without fairness restriction.  We will apply this result in our future paper \cite{dardayasudasemicomm} to solve the {\it inverse Galois problem} for so-called {\it semicommutative} finite group schemes.
\subsection{Overview}We present basic concepts and results of each of sections. We fix a non-trivial finite \' etale tame $F$-group scheme~$G$. 
\subsubsection{} In Section \ref{sectionheights}, a theory of heights is proposed. For that purpose a finite \' etale tame scheme~$G_*$ is defined. When~$G$ is constant its $F$-points are ``$F$-conjugacy classes" from \cite{Malle}. For every place~$v$ such that~$G$ is unramified and tame at~$v$, we have a {\it residue} map $\psi_v:BG(F_v)\to G_*(\Fv)\subset G_*(\oF)$. The scheme $G_*$ and the residue map are very convenient tools to define heights. We define a {\it counting function} to be $\Gamma_F$-invariant function $c: G_*(\oF)\to\RR_{\geq 0}$ with $c(x)=0$ if and only if $x=1$. One easily defines its~$a$ and~$b$ invariants (the earlier mentioned~$a$ and~$b$ invariants of heights are in fact invariants of counting functions). Heights are now defined to be the functions $H:BG(F)\to\RR_{>0}$ for which $v$-adic components~$H_v:BG(F_v)\to\RR_{>0}$ for almost all~$v$ are defined by $H_v(x)=q_v^{c(\psi_v(x))}$. At the end of the section, we state our conjecture. 
\subsubsection{} \label{detailscommutative} Section \ref{sectioncommutativecase} is devoted to the proof of Theorem \ref{equintro}. Let us sketch it. It uses harmonic analysis. The idea originates from a method used to count rational points on equivariant compactifications of algebraic groups, developed initially by Batyrev and Tschinkel for the case of anisotropic tori \cite{aniso}. In the context of counting field extensions, harmonic analysis was first used by Frei, Loughran and Newton in \cite{hassenormprin} (they also used it in \cite{prescnorms}). The method has been used in \cite{odorney} to count elements of $BG(F)$ when~$F$ is a number field and for so-called ``Frobenian invariants". It seems to us that there may be a gap in their analysis: in the proof of Theorem 1.1, the claimed meromorphic extension for non-trivial characters to the left of~$1/a_{\inv}(\mathcal L)$ in their notation does not appear to be automatic.  Without the meromorphic extension for non-trivial characters, one would only obtain an upper bound, such as~$o(B^{a(c)+\epsilon})$. We were told by Alberts that the gap should not be hard to fill by using a result of \cite{statfirstg}, which is based on a result of \cite{prescnorms}. For the sake of completeness, we give a proof using different arguments.

By using Tauberian results, the task becomes to analyse the Dirichlet series $Z(s):=\frac{1}{\# G(F)}\sum_{x\in BG(F)} H(x)^{-s}$ for $s\in\CC$. It follows from polynomial bounds on the number of extensions of fixed degree and of bounded discriminant (c.f. \cite{SchmidtNF} and \cite[Lemma 2.4]{Gmethodant}) that the series converges for $\Re(s)\gg 0$ and defines a holomorphic function in a complex halfplane.  One has a map $$i:BG(F)\to BG(\AAA):=\sideset{}{'}\prod_vBG(F_v),$$ where the restricted product is taken with respect to the unramified cohomology. It follows from the Poitou-Tate exact sequence that the map~$i$ is of finite kernel and of discrete, closed and cocompact image.  By applying the Poisson formula for the inclusion $i(BG(F))\subset BG(\AAA),$ the function~$Z$ is written as a finite sum of Fourier transforms of the extended height $H:BG(\AAA)\to\RR_{>0}$ at automorphic characters (i.e. the ones vanishing at~$i(BG(F))$).   
We ``approach" these Fourier transforms by $L$-functions of certain Galois representations. Namely, the representations are ``twists" by characters, of  restricted Galois representations induced from Galois submodule of~$G_*(\oF)$ where $c:G_*(\oF)\to\RR_{\geq 0}$ attains its minimal strictly positive value. The analytic properties of~$Z$ are now deduced from analytic properties of $L$-functions. 
%
\subsection{Acknowledgements} This work was supported by JSPS KAKENHI Grant Number JP18H01112. This work has been done during a post-doctoral stay of the first named author at Osaka University. During the stay, he  was supported by JSPS Postdoctoral Fellowship for Research in Japan. The authors would like to thank Antoine Chambert-Loir, Daniel Loughran and Julian Demeio for useful comments. 
\subsection{Conventions} We will use notation~$F$ for a global field. We denote by~$M_F$ (respectively,~$M_F^0$ and~$M_F^\infty$) the set of its places (respectively, of finite places and of infinite places). 

We fix algebraic closures of~$F$ and of~$F_v$ for $\vMF$ and embeddings of the algebraic closure of~$F$ in each of the algebraic closures of~$F_v$. We denote by~$\overline F$ and for $\vMF$ by $\overline{F_v}$ the separable closure of~$F$ in and~$F_v$ in the chosen algebraic closures.

If~$v$ is a finite place of~$F$, we denote by~$\Ov$ the ring of integer of~$F_v$, by~$\kappa_v$ the residue field at~$v$ and by~$q_v$ its cardinality.  In the function field case, we write~$q$ for the cardinality of the field of constants of~$F$. The notation~$\Gamma_F$ and~$\Gamma_v$ will be used to denote the absolute Galois group of~$F$ and~$F_v$, respectively.

Let $f, g:\RR_{\geq 0}\to\RR_{\geq 0}$ be two functions, such that for $B\gg 0$ one has that $g(B)\neq 0$. We write $f\sim_{B\to\infty}g$ if $\lim_{B\to\infty}\frac{f(B)}{g(B)}=1$. We write $f\asymp_{B\to\infty}g$ if there are constants $C_1, C_2>0$ such that for every~$B$ big enough one has that $ C_1 g(B)\leq f(B)\leq C_2 g(B)$.
\section{Heights} \label{sectionheights} In this section we will define heights that we will be using to count elements of~$BG(F)$.
\subsection{Cohomologies}\label{equivalenceofcat} We start by defining cohomological groups that we work with. 
\subsubsection{}Let~$M$ be a profinite group and let~$J$ be a finite $M$-group (that is, a finite group endowed with a continuous $M$-action). We denote by $H^1(M, J)$ the pointed set given by the corresponding non-abelian degree~$1$ cohomology \cite[Chapter I, Section 5.2]{Cohomologiegalois}. It is the set of finite sets endowed with compatible~$M$ and~$J$ actions modulo $M$-equivariant bijections such that the map $$J\times A\xrightarrow{(g,x)\mapsto (g\cdot x,x)} A\times A$$ is a bijection. The distinguished element of~$H^1(M,J)$ is given by the class of~$J$, with the action being given by the left multiplication. 
The pointed set $H^1(M, J)$ also allows the following description: it is the pointed set \begin{align*}Z^1(M,J):=\{f:M\to J\text{ continuous}|\forall\gamma,\delta\in M:f(\gamma\delta)=(f(\gamma))(\gamma\cdot f(\delta))\}
\end{align*} modulo the equivalence relation $$f\sim f'\iff \exists g\in J: \forall \gamma\in M: f'(\gamma)=g^{-1}f(\gamma)(\gamma\cdot g).$$
In the case~$J$ is commutative, one has that $Z^1(M,J)$ and $H^1(M, J)$ have structures of abelian groups.
\subsubsection{}Let~$S$ be a locally Noetherian connected scheme and let~$s$ be a geometric point of~$S$. Let $\et(S)$ be the category having finite \' etale $S$-schemes for its objects and \' etale morphisms of $S$-schemes for its morphisms. We let $\pi_1(S, s)$ be the fundamental group of~$S$ at~$s$ and let $\pi_1(S,s)-set$ be the category of {\it finite} sets endowed with a continuous $\pi_1(S,s)$-action. By \cite[\' Expos\' e V, Section 7]{SGAj}, one has an equivalence $$F_S:\et (S)\to \pi_1(S,s)-set.$$ 
Let~$S'$ be another locally Noetherian connected scheme endowed with a morphism $f:S'\to S$, and~$s'$ its geometric point such that $f(s')=~s$. We have a continuous homomorphism $\pi_1(S', s')\to \pi_1(S,s)$. Let us denote by~$F_S^{-1}$ a quasi-inverse of~$F_S$ and by $f^*:\et(S)\to\et(S')$ the functor given by the base change to~$S'$. The functor $$F_S\circ f^*\circ F_S^{-1}:\pi_1(S,s)-set\to\pi_1(S', s')-set$$ is isomorphic to the functor given by the the homomorphism $\pi_1(S', s')\to \pi_1(S,s)$. 

Let~$G$-be a finite \' etale group scheme over~$S$. We note a canonical identification $$G(S)=F_S(G)^{\pi^1(S,s)}\hspace{2cm} (=:H^0(S, G)).$$ A (left) $G$-torsor over~$S$ is a scheme~$P$ endowed with an {\it fppf} morphism to~$S$ and  with a (left) $G$-action such that the morphism $G\times _SP\to P\times _SP$ given by $(g, p)\mapsto (g\cdot p, p)$ is an isomorphism. A morphism of $G$-torsors $P\to P'$ is a $G$-equivariant morphism $P\to P'$. The trivial~$G$-torsor is $P=G$ endowed with the action given by the left multiplication.
\begin{mydef}
We denote by $BG(S)$ the pointed set of isomorphism classes of $G$-torsors. If $S=\Spec(A)$ for a ring~$A$, we use also notation $BG(A)$.
\end{mydef} 
 As~$G$ is assumed \' etale, by \cite[Chapter III, Proposition 4.2]{Milne}, one has for a $G$-torsor~$P$ that $P\in\et (S)$. Obviously, any \' etale morphism is automatically {\it fppf}, thus~$P\in\et(S)$ is a $G$-torsor if and only if $G\times_S P\to P\times_S P$ is an isomorphism, and using the equivalence~$F_S$, if and only if $F_S(G)\times F_S(P)\to F_S(P)\times F_S(P)$ is an isomorphism in $\pi_1(S,s)-set$. We have hence established a canonical identification of~$BG(S)$ with the set $H^1(\pi_1(S,s), F_S(G))$. 
\subsubsection{} We apply results of the previous paragraph. Let~$F$ be a global field. Let $S=\Spec(A)$, where $$A\in\{F\}\cup\cup_{\vMFz}\{\Ov\}\cup\cup_{\vMF}\{F_v\}.$$ 
By~$s$, we will denote each time a geometric generic point of a scheme~$S$. Note that for every \' etale morphism $S'\to S,$ every point in the geometric fiber over~$s$ is defined over the separable closure of the function field of~$S$, so that $F_{S}(S')$ is given by the points in the separable closure of the function field.

By \cite[\' Expos\' e V, Proposition 8.1]{SGAj}, one has an identification $\pi^1(\Spec(F), s)=\Gamma_F$ and for finite places~$v$, an identification $\pi^1(\Spec(F_v), s)=~\Gamma_v$. For the last identification, by \cite[\' Expos\' e V, Proposition 8.2]{SGAj}, the group $\pi^1(\Spec(\Ov), s)$ is the quotient group of $\pi^1(\Spec(\Fv), s)$ corresponding to the quotient group~$\Gamma_v^{\un}$ of~$\Gamma_v$. 

Let~$G$ be a finite \' etale $S$-group scheme. 
When $S=\Spec(\Ov)$, we obtain that the map of pointed sets 
$$BG(\Ov)\to BG(F_v)$$ identifies with the unramified cohomology $$H^1(\Gamma_v^{\un}, G(\overline {F_v}))\hookrightarrow H^1(\Gamma_{F_v}, G(\overline {F_v})).$$ (The last map is an inclusion by \cite[Chapter I, \S 5.8 (a) and Chapter I \S 2.6 (b)]{Cohomologiegalois}.) When $S=\Spec(F)$, for every place~$v$, the inclusion $\overline F\hookrightarrow\overline F_v$, yields a map $$BG(F)\to BG(F_v)$$ which corresponds to the map on cohomologies induced by the inclusion $\Gamma_v\subset\Gamma_F$. 
\subsection{Definition of $G_*$}For an \' etale $F$-group scheme~$G$, we are going to define a pointed \' etale scheme~$G_*$. If~$G$ is commutative, then~$G_*$ has a structure of a commutative group scheme. When~$G$ is a constant group scheme, the $\overline F$-points of~$G_*$ correspond to $F$-conjugacy classes from \cite{Malle}. 
\subsubsection{}

Let $M$ be a profinite group, let $N$ be a profinite $M$-group and let $J$ be a finite $M$-group. 
We denote by $\Hom(N,J)$ the group of (automatically continuous) homomorphisms $N\to J$. The group $\Hom(N,J)$ is endowed with a canonical structure of an $M$-group:
$$\gamma\cdot f:=\gamma\circ f\circ\gamma^{-1}.$$
The group $\Hom(N,J)$ is also endowed with a continuous action of~$J$ given by conjugation: $$j\cdot f:= (x\mapsto (j)(f(x))j^{-1})\hspace{2cm}j\in J, f\in\Hom(N,J).$$
Note that: 
\begin{align*}\gamma\cdot (j\cdot f)&=\gamma \cdot (x\mapsto (j)(f(x))j^{-1})\\&=\big(x\mapsto\gamma\big((j)(f(\gamma^{-1}\cdot x))j^{-1}\big)\big)\\
&=\big(x\mapsto\gamma(j)\gamma(f(\gamma^{-1}\cdot x))\gamma(j)^{-1}\big)\\
&=(\gamma\cdot j)\cdot (\gamma\cdot f),
\end{align*}hence, in particular the action of~$M$ preserves the orbits of~$J$. We deduce that there is a continuous $M$-action on $$\Conj(N,J):=\Hom(N, J)/J$$given by $$\gamma\cdot(fJ)=(\gamma\cdot f)J.$$ The set $\Conj(N,J)$ admits a distinguished element which is fixed by~$M$. The fiber of the distinguished element in $\Conj(N,J)$ is the distinguished element in $\Hom(N,J)$. If~$J$ is assumed commutative, one has that~$\Hom(N,J)=\Conj(N,J)$ admits a structure of an $M$-module. 

Note that if $N\to N'$ is a homomorphism of profinite groups, the canonical map $\Hom(N',J)\to\Hom(N,J)$ is $J$-equivariant and we have a map of pointed sets$$\Conj(N',J)\to\Conj(N, J).$$ If $N\to N'$ is a homomorphism of profinite $M$-groups, one has that the map $\Conj(N',J)\to\Conj(N, J)$ is a morphism of $M$-pointed sets. Similarly, if $J\to J'$ is a homomorphism of finite groups, the canonical map $\Hom(N,J)\to\Hom(N, J')$ is $(J\to J')$-equivariant and we have a map of pointed sets $$\Conj(N, J)\to\Conj(N, J').$$ If $J\to J'$ is a homomorphism of finite $M$-groups, the map $\Conj(N, J)\to\Conj(N, J')$ is a morphism of $M$-pointed sets. 
\subsubsection{}\label{paragtwist} Let now~$M$ be a profinite group and let~$J$ be a finite $M$-group. Let $\sigma\in Z^1(M, J)$. We denote by~$\prescript{}{\sigma}J$ the twist of~$J$ by~$\sigma$ \cite[Section 5.4, Example 2]{Cohomologiegalois}; it is the abelian group~$J$ endowed with with the $\Gamma_F$-action $$\gamma\cdot_{\sigma} f:=(\sigma(\gamma))(\gamma(f))(\sigma(\gamma)^{-1}).$$ 
By \cite[Chapter II, Proposition 35 bis]{Cohomologiegalois}, one has a pointed bijection $$\Lambda_{\sigma}:Z^1(M, \sJ)\to Z^1(M, J),\hspace{1cm}\beta\mapsto \beta\cdot\sigma,$$ which induces a pointed bijection $$\lambda_{\sigma}:H^1(M, \sJ)\to H^1(M, J).$$
Let us recall several properties of twists.
\begin{lem}\label{listproptwist} Let~$\sigma\in Z^1(M,J)$ and let~$\tau\in Z^1(M,\prescript{}{\sigma}J)$.
\begin{enumerate}
\item One has that $\prescript{}{\tau}(\prescript{}{\sigma}J)=\prescript{}{\Lambda_{\sigma}(\tau)}(J).$
\item Let $\sigma'$ be cohomologous to~$\sigma$ and let $k\in J$ be such that $$\sigma'(\gamma)=(k^{-1})(\sigma(\gamma))(\gamma(k))$$ for every $\gamma\in M$. The isomorphism of abelian groups $$\prescript{}{\sigma'}G\to\prescript{}{\sigma} G\hspace{1cm}g\mapsto kgk^{-1}$$is $M$-equivariant (i.e. is an isomorphism of $M$-groups).
\item One has that $$ \Lambda_{\Lambda_{\sigma}(\tau)}=\Lambda_{\sigma}\circ\Lambda_{\tau}$$ and that  $$ \lambda_{\Lambda_{\sigma}(\tau)}=\lambda_{\sigma}\circ\lambda_{\tau}.$$
\item Suppose that $f:N\to M$ is a continuous homomorphism of topological groups. The diagrams  
\[\begin{tikzcd}
	{Z^1(M, \sJ)} & {Z^1(N,\sJ)=Z^1(N, \prescript{}{\sigma\circ f}J)} \\
	{Z^1(M, J)} & {Z^1(N, J)}
	\arrow[from=1-1, to=1-2]
	\arrow["{\Lambda_{\sigma}}"', from=1-1, to=2-1]
	\arrow["{}"', from=2-1, to=2-2]
	\arrow["{\Lambda_{\sigma\circ f}}", from=1-2, to=2-2]
\end{tikzcd}\]
and 
\[\begin{tikzcd}
	{H^1(M, \sJ)} & {H^1(N,\sJ)=H^1(N, \prescript{}{\sigma\circ f}J)} \\
	{H^1(M, J)} & {H^1(N, J)}
	\arrow[from=1-1, to=1-2]
	\arrow["{\lambda_{\sigma}}"', from=1-1, to=2-1]
	\arrow["{}"', from=2-1, to=2-2]
	\arrow["{\lambda_{\sigma\circ f}}", from=1-2, to=2-2]
\end{tikzcd}\]
are commutative.
\item Let~$K$ be another $M$-group and let $h:J\to K$ be an $M$-equivariant homomorphism. The canonical homomorphism of finite groups $\sJ=J\xrightarrow{h} K=\prescript{}{h\circ\sigma}K$ is $M$-equivariant. The diagrams
\[\begin{tikzcd}
	{Z^1(M, \sJ)} & {Z^1(M,\prescript{}{h\circ\sigma}K)} \\
	{Z^1(M,J)} & {Z^1(M,K)}
	\arrow[from=1-1, to=1-2]
	\arrow["{\Lambda_{\sigma}}"', from=1-1, to=2-1]
	\arrow[from=2-1, to=2-2]
	\arrow["{\Lambda_{h\circ\sigma}}", from=1-2, to=2-2]
\end{tikzcd}\]
and
\[\begin{tikzcd}
	{H^1(M, \sJ)} & {H^1(M,\prescript{}{h\circ\sigma}K)} \\
	{H^1(M,J)} & {H^1(M,K)}
	\arrow[from=1-1, to=1-2]
	\arrow["{\lambda_{\sigma}}"', from=1-1, to=2-1]
	\arrow[from=2-1, to=2-2]
	\arrow["{\lambda_{h\circ\sigma}}", from=1-2, to=2-2]
\end{tikzcd}\]
are commutative.
\end{enumerate}
\end{lem}
\begin{proof}
\begin{enumerate}
\item  Let $\gamma\in M$ and let $g\in \prescript{}{\tau}(\prescript{}{\sigma}J)=J$. For the canonical action of~$M$ on~$\prescript{}{\tau}(\prescript{}{\sigma}J),$ we have that 
\begin{align*}\gamma\cdot g&=(\tau(\gamma)) (\gamma\cdot_{\sigma}g)(\tau(\gamma))^{-1}\\&=(\tau(\gamma)) \cdot (\sigma(\gamma))\cdot (\gamma(g))\cdot(\sigma(g))^{-1}) \cdot (\tau(\gamma))^{-1}\\
&=(\Lambda_{\sigma}(\tau)(\gamma))(\gamma(g))(\Lambda_{\sigma}(\tau)(\gamma))^{-1},
\end{align*}
which is precisely the element $\gamma\cdot g$ for the canonical action of~$M$ on~$\prescript{}{\Lambda_{\sigma}(\tau)}(J)$. 
\item Let $\gamma\in M$ and let $g\in G$. We have that\begin{align*}(k)(\gamma\cdot_{\sigma'})(k^{-1})\hskip-2cm&\\&=(k)\big(\sigma'(\gamma)\big)(\gamma(g))\big(\sigma'(\gamma))^{-1}\big)(k^{-1})\\
&=(k)\bigg((k^{-1})(\sigma(\gamma))(\gamma(k))\bigg)(\gamma(g))\bigg((\gamma(k)^{-1})(\sigma(\gamma))^{-1}(k)\bigg)(k^{-1})\\
&=(\sigma(\gamma))(\gamma(k))(\gamma(g))(\gamma(k^{-1}))(\sigma(\gamma))^{-1}\\
&=\gamma\cdot_{\sigma}(kgk^{-1}),
\end{align*}
and the claim follows.
\item Let $\theta\in Z^1(M,\prescript{}{\tau}(\prescript{}{\sigma}J))$. We have that $$\Lambda_{\Lambda_{\sigma}(\tau)}(\theta)=\theta\cdot \Lambda_{\sigma}(\tau)=\theta\cdot \tau\cdot\sigma\in Z^1(M, J).$$ On the other side $$\Lambda_{\sigma}(\Lambda_{\tau}(\theta))=\Lambda_{\sigma}(\theta\cdot\tau)=\theta\cdot\tau\cdot\sigma.$$ The first claim follows. The second claim follows immediately from the first claim.
\item The first diagram is commutative, because the image of $\theta\in Z^1(M, \sJ)$ under the composite of left vertical and upper horizontal maps is $\Lambda_{\sigma\circ f}(\theta\circ f)=(\theta\circ f)\cdot (\sigma\circ f),$ which is precisely $(\theta\cdot \sigma)\circ f=\Lambda_{\sigma}(\theta)\circ f.$ The commutativity of the second diagram follows immediately from the commutativity of the first diagram.
\item Let $j\in J$ and let $\gamma\in M$. We have that 
\begin{align*}
h(\gamma\cdot_{\sigma}j)&=h\big((\sigma(\gamma))\cdot(\gamma(j))\cdot(\sigma(\gamma))^{-1}\big)\\
&=h(\sigma(\gamma))\cdot h(\gamma(j))\cdot (h(\sigma(\gamma)))^{-1}\\
&=((h\circ\sigma)(\gamma))\cdot\gamma(f(j))\cdot((h\circ\sigma)(\gamma))^{-1}\\
&=\gamma\cdot_{h\circ\sigma} f(j),
\end{align*}
and the first claim follows. If $\theta\in Z^1(M, \sJ)$, its image under the composite of the left vertical and upper horizontal maps is given by $\Lambda_{h\circ\sigma}(h\circ \theta)=(h\circ\theta)\cdot (h\circ \sigma).$ This is equal to $h\circ (\theta\cdot\sigma)=h\circ(\Lambda_{\sigma}(\theta))$ and the commutativity of the first diagram follows. The commutativity of the second diagram is now immediate.
\end{enumerate}
\end{proof}
\subsubsection{}We now turn to~$\Gamma_F$-groups. 
When $k\geq 1$ is an integer, we denote by~$\muu_k$ the group of~$k$-th roots of unity in~$\overline{F}$. It has a structure of~$\Gamma_F$-module: $$\Gamma_F\times\muu_k\to\muu_k\hspace{1cm}(\gamma,\xi)\mapsto \xi^{\chi_{\cycl}(\gamma)},$$where $\chi_{\cycl}:\Gamma_F\to\prod_{\ell\neq \charr(F)}\ZZ_{\ell}^{\times}$ is the cyclotomic character \cite[Definition 7.3.6]{CohomologyNF}. Let~$J$ be a finite $\Gamma_F$-group.
\begin{mydef}
Let $e=e(J)$ be the exponent of~$J$. 
We set $$J_*:=\Conj(\muu_{e},J)$$ and we endow it with the continuous $\Gamma_F$-action as above. If~$J$ is abelian, then~$J_*$ is understood to be endowed with a structure of a commutative $\Gamma_F$-module.
\end{mydef}
\begin{rem}
\normalfont
When~$J$ is abelian, the $\Gamma_F$-module~$J_*$ is precisely the Tate twist $J(-1)$ \cite[Section 7.8]{cohomologicalinvariants}.
\end{rem}
\begin{lem}\label{slabalema}
\begin{enumerate}
\item Suppose that $J$ is endowed with a trivial~$\Gamma_F$-action.  Let $\Conj_F(J)$ be the set of $F$-conjugacy classes of~$J$ (\cite[Section 2]{Malle})\footnote{In \cite{Malle}, Malle works only for the case~$F$ is a number field. The definitions and results that we use, are generalized to the case of a global field in an evident way.}. A choice of be a primitive $e=e(J)$-th root of unity $\xi$ in $\overline F$ induces a map of pointed sets $$\lambda_{\xi}:\Hom(\muu_e, J)\to J\to\Conj_F(J),\hspace{1cm} f\mapsto [f(\xi)]$$ where $[g]$ denotes the $F$-conjugacy class of~$g$. The map $\lambda_{\xi}$ is surjective, of trivial kernel, $J$-invariant, $\Gamma_F$-invariant and induces a bijection $$J_*/\Gamma_F\xrightarrow{\sim}\Conj_F(J).$$ If~$r\geq 1$ is an integer such that $\gcd(r,e)=1,$ then for $y\in \Hom(\muu_e, J)$ one has that $\lambda_{\xi}(y)=\lambda_{\xi^r}(y)^r,$ where the map of pointed sets $\Conj_F(J)\to\Conj_F(J)$ denoted by $x\mapsto x^r$ is the map induced from $J$-equivariant and $\Gamma_F$-equivariant automorphism $J\to J$ given by $j\mapsto j^r$.  
\item Suppose that $J=~\muu_k$ for some integer~$k\geq 1$. The~$\Gamma_F$-action on~$(\muu_k)_*$ is trivial.
\end{enumerate}
\end{lem}
\begin{proof}
\begin{enumerate}
\item Note that $f\mapsto f(\xi)$ is a group isomorphism $$\Hom(\muu_e, J)\xrightarrow{\sim} J.$$ (It is an injective group homomorphism, which is surjective, because an element $g\in J$ of order~$k$ is the image of the composite homomorphism $\muu_e\to\muu_k\to J$ which is given by $\xi\mapsto \xi^{e/k}\mapsto g$). We endow~$J$ with a $\Gamma_F$-action using the isomorphism. For this action, one has that $\gamma\cdot g=g^{\chi_{\cycl}(\gamma)^{-1}}$, thus, this action is the opposite action of the one given in \cite[Section 2]{Malle}. Hence, the~$J$-invariant, surjective map of pointed sets $J\to \Conj_F(J)$ with the trivial kernel is~$\Gamma_F$-invariant.  Let us prove the last claim.  The map $f\mapsto f(\xi^r)$ factorizes as the composite of the maps $\Hom(\muu_e,J)$ given by $f\mapsto f(\xi)$ and the map $J\to J$ given by $x\mapsto x^r$. Hence, $\lambda_{\xi^r}(f)=[f(\xi^r)]=[f(\xi)^r]=[f(\xi)]^r=\lambda_\xi(f)^r,$ as claimed
\item This is a well known fact, but we include proof for completeness. Let~$\xi$ be a primitive~$k$-th root of unity in~$\oF$. A homomorphism $f:\muu_k=\muu_{e(\muu_k)}\to \muu_k$ is determined by the image on~$\xi$. We have that $$(\gamma\cdot f)(\xi)=\gamma(f(\gamma^{-1}(\xi)))=f(\xi^{\chi_{\cycl}(\gamma)^{-1}})^{\chi_{\cycl}(\gamma)}=f(\xi),$$i.e. $\gamma\cdot f=f$.
\end{enumerate}
\end{proof}
For positive integers $d,k$ with $d|k$, one has a canonical $\Gamma_F$-equivariant surjective homomorphism $$q_{k,d}:\muu_k\to\muu_d\hspace{1cm}x\mapsto x^{k/d}.$$
\begin{lem}\label{funcinj}
Let $\iota:N\hookrightarrow J$ be an injective $\Gamma_F$-equivariant homomorphism (in particular $e(N)|e(J)$). The map $$\Hom(\muu_{e(N)}, N)\to\Hom(\muu_{e(J)},J)\hspace{1cm}f\mapsto \iota\circ f\circ q_{e(J), e(N)}$$ is an injective homomorphism which is $\iota$-equivariant and $\Gamma_F$-equivariant. The induced map of pointed sets $\iota_*:N_*\to J_*$ is $\Gamma_F$-equivariant and of trivial kernel.
\end{lem}
\begin{proof}
The map $\Hom(\muu_{e(N)}, N)\to\Hom(\muu_{e(J)},J)$ is the composite of the injective homomorphisms \begin{equation}\label{compnj}\tag{*}\Hom(\muu_{e(N)}, N)\to\Hom(\muu_{e(J)}, N)\to\Hom(\muu_{e(J)}, J), \end{equation} hence is an injective homomorphism. The first map in (\ref{compnj}) is $N$-invariant. The image of $f\in\Hom(\muu_{e(J)},N)$ for the second map in (\ref{compnj}) is $\iota\circ f$, while for $g\in N$, the image of $g\cdot f=(x\mapsto (g)f(x)(g)^{-1})$ is given by $$\big(x\mapsto \iota((g)f(x)(g)^{-1})\big)=\iota(g)\cdot (\iota\circ f).$$ It follows that the second map in (\ref{compnj}) is $\iota$-equivariant. Hence, the canonical map $\Hom(\muu_{e(N)},N)\to\Hom(\muu_{e(J)},J)$ is $\iota$-equivariant. By the fact that $\muu_{e(J)}\to\muu_{e(N)}$ is $\Gamma_F$-equivariant, it follows that the first map in (\ref{compnj}) is $\Gamma_F$-equivariant and by the fact that $\iota$ is $\Gamma_F$-equivariant, it follows that the second map in (\ref{compnj}) is $\Gamma_F$-equivariant. Hence, $\Hom(\muu_{e(N)}, N)\to\Hom(\muu_{e(J)}, J)$ is $\Gamma_F$-equivariant. It follows that the map $\iota_*:N_*\to J_*$ is $\Gamma_F$-equivariant. The map $\Hom(\muu_{e(N)}, N)\to\Hom(\muu_{e(J)}, J)\to J_*$ is of the trivial kernel as the composite of such maps, hence is such $N_*\to J_*$. The statement is proven. 
\end{proof}
We now verify that twisting does not change~$J_*$ and some other properties. 
\begin{lem}\label{twistsandstar} Let $\sigma\in Z^1(\Gamma_F, J)$.
\begin{enumerate}
\item One has an equality of $\Gamma_F$-pointed sets $(\prescript{}{\sigma}J)_*=J_*.$
\item Let $R\subset J$ be a $\Gamma_F$-invariant subgroup. Let $\tau\in Z^1(\Gamma_F, R)$. The canonical map $R_*\to J_*$ identifies with the map $(\prescript{}{\tau}R)_*\to (\prescript{}{\tau}J)_*.$
\item Let~$\sigma'\in  Z^1(\Gamma_F, J)$ be cohomologous to~$\sigma$ and let $k\in J$ be such that for every $\gamma\in\Gamma_F$ one has that $\sigma'(\gamma)=(k^{-1})(\sigma(\gamma))(k).$ The isomorphism $$f_k:\prescript{}{\sigma'}J\to\sJ\hspace{1cm}j\mapsto kjk^{-1},$$which is $\Gamma_F$-equivariant by Part (2) of Lemma \ref{listproptwist}, induces the identity map $(f_k)_*:J_*=(\prescript{}{\sigma'}J)_*\xrightarrow{\sim} (\sJ)_*=J_*$. If~$N$ is a $\Gamma_F$-invariant subgroup of~$\prescript{}{\sigma'}J$, the diagram 
\[\begin{tikzcd}
	{N_*} \\
	{(f_k(N))_*} & {J_*}
	\arrow[from=1-1, to=2-1]
	\arrow[from=1-1, to=2-2]
	\arrow[from=2-1, to=2-2]
\end{tikzcd}\] is commutative in which the vertical map is $\Gamma_F$-equivariant bijection of pointed sets.
\end{enumerate}
\end{lem}
\begin{proof}
\begin{enumerate}
\item Clearly, one has an equality of pointed sets $\Hom(\muu_e, \prescript{}{\sigma}J)=\Hom(\muu_e, J)$ and hence of pointed sets $(\prescript{}{\sigma}J)_*=J_*$. We prove that the Galois actions are identical. Let~$\gamma\in\Gamma_F$ and let~$j\in (\prescript{}{\sigma}J)_*$. Let~$\widetilde j\in \Hom(\muu_e, \prescript{}{\sigma}J)$ be a lift of~$j$. We have that $$\gamma\cdot_{\sigma} j=[(\sigma(\gamma))(\gamma(\widetilde j))(\sigma(\gamma))^{-1}]=[\gamma(\widetilde j)]=\gamma\cdot j,$$where 
$[-]$ stands for the equivalence class in~$(\prescript{}{\sigma}J)_*=J_*$ of an element. The statement is proven.
\item This follows from Part (1) and the fact that the canonical map $R_*\to J_*$ is defined purely in the terms of groups (and does not depend on Galois actions). 
\item Let $j\in (\prescript{}{\sigma'}J)_*=J_*$ and let $\widetilde j\in \Hom(\muu_e,J)$ be its lift. One has that~$(f_k)_*(j)$ is the image of~$(k)(\widetilde j)(k)^{-1}$ under the quotient map $\Hom(\muu_e, J)\to J_*=(\prescript{}{\sigma}J)_*$ and this equals to the image of~$\widetilde j$ under the quotient map, hence $(f_k)_*(j)=j$. Let us prove the second claim. Let $g\in N_*$ and let~$\widetilde g\in N$ be its lift. The image of~$g$ in~$J_*$ equals to the class of~$\widetilde g\in J$. On the other side, the image of~$g$ in~$(f_k(N))_*$ equals to the image of $k\widetilde gk^{-1}$ in~$f_k(N)$. The map $(f_k(N))_*\to J_*$ takes the class of $k\widetilde gk^{-1}$ to the class of~$k\widetilde gk^{-1}$ in~$N_*$, which coincides with the class of $\widetilde g$. The second claim is proven.
\end{enumerate}
\end{proof}
\subsubsection{}Let~$G$ be a finite \' etale $F$-group scheme. Set $J=G(\oF)$ and endow it with the usual $\Gamma_F$-action. The~$\Gamma_F$-pointed set~$J_*$ defines a pointed \' etale scheme~$G_*$. If~$G$ is commutative, then $G_*$ has a structure of a commutative group scheme.  If $\sigma\in Z^1(\Gamma_F, G(\oF))$, we denote by~$\sG$ the finite \' etale tame group scheme defined by the $\Gamma_F$-group $\prescript{}{\sigma}(G(\oF)).$
\subsection{Counting function}
In this paragraph we define counting functions. They are used to define heights. When~$F=\QQ$ and~$G$ is constant, the definition coincides with the one of \cite[Section 4.2]{Gmethodant}. 
\subsubsection{}In this paragraph, we will define {\it discriminant counting function}. It ``induces" discriminants of $G$-torsors as will see in Subsection~\ref{discofgtorsor}.
\begin{mydef}
Let $J$ be a non-trivial finite $\Gamma_F$-group. A counting function is $\Gamma_F$-invariant function $$c:J_*\to\RR_{\geq 0},$$ which satisfies $c(g)=0$ if and only if~$g=1_{J_*}$ is distinguished in~$J_*.$ We define $$a(c):=\big(\min_{g\in J_*-\{1_{J_*}\}}\{c(g)\}\big)^{-1}\in\RR_{>0}.$$ The function~$c$ is said to be normalized if $a(c)=1.$
\end{mydef}
For any counting function $c:J_*\to\RR_{\geq 0}$, the function 
$$J_*\to\RR_{\geq 0},\hspace{1cm} g\mapsto \frac{c(g)}{a(c)}$$ is a normalized counting function.
\begin{mydef}
For $q\in\RR_{\geq 0}$, we denote by $(J_*)_q$ the $\Gamma_F$-invariant subset of the pointed set~$J_*$ given by $(J_*)_q:=c^{-1}(q).$ We define $$b(c):=\#  \big(((J_*)_{a(c)^{-1}})/\Gamma_F\big).$$
\end{mydef}
In the rest of paragraph we present several examples. 
 We have a function 
\begin{equation}
\Hom(\muu_e,J)\to\QQ_{\geq 0},\hspace{1cm} f\mapsto [J:f(\muu_e)]\cdot (\# f(\muu_e)-1).
\label{discount}\tag{*}
\end{equation}The function is $J$-invariant and $\Gamma_F$-invariant as the orders and indices of subgroups are preserved under group automorphisms. Moreover, it takes the value~$0$ at the trivial homomorphism $\muu_e\to J$ and takes strictly positive values elsewhere.
\begin{mydef}
Let $c_{\Delta}:J_*\to\QQ_{\geq 0}$ be the counting function induced by the $J$-invariant function (\ref{discount}). We call it the discriminant counting function.
\end{mydef}
\subsubsection{} In this paragraph, we give another example of a counting function. The corresponding height is used in the statement of the Malle conjecture.

Let~$J$ be a non-trivial constant $\Gamma_F$-group. We suppose that there is a transitive embedding $\iota:J\hookrightarrow \mathfrak S_n$ for certain $n\geq 1$, where~$\mathfrak S_n$ stands for the group of permutations of the set $\{1\doots n\}$. For $g\in \Conj_F(J),$ let $\widetilde g\in J$ be a lift of $g$. We set
\begin{align*}\ind^{\iota}:\Conj_F(J)&\to\ZZ_{\geq 0}\\g&\mapsto\ind(\iota(\widetilde g)):=n-\# \{\textnormal{orbits of }\iota(\widetilde g)\textnormal{ in \{1\doots n\}}\}.
\end{align*} The definition does not depend on the choice of $\widetilde g$ and for $g\neq 1\in\Conj_F(J)$ one has $\ind(g)\neq 1$, as stated in \cite[Section 2]{Malle}. Note that if $r\geq 1$ is an integer such that $\gcd(r,e)=1$, then $\ind^{\iota} (g)=\ind^{\iota}(g^r)$. For a primitive $e=e(J)$-th root of unity~$\xi,$ we have a $\Gamma_F$-invariant surjection of pointed sets $\overline\lambda_\xi:J_*\to \Conj_F(J)$ induced from $J$-invariant and $\Gamma_F$-invariant surjection $\lambda_\xi:\Hom(\muu_e,J)\to \Conj_F(J)$ which is defined in Lemma \ref{slabalema}. Moreover, one has that $\overline\lambda_{\xi^r}=\overline\lambda_{\xi}^r$ and that $\overline{\lambda_{\xi}}$ is of trivial kernel. It follows from above that the map $$c_{{\iota}}:J_*\to\QQ_{\geq 0}\hspace{1cm}c_{{\iota}}:=\ind^{\iota}\circ\overline{\lambda}_\xi$$ is a counting function and that does not depend on the choice of $\xi$. 
It is immediate that one has that $$a(c_{{\iota}})=a'(J\leq \mathfrak S_n),$$ and $$b(c_{{\iota}})=b'(F,J\leq \mathfrak S_n),$$where $a'(-)$ and $b'(-,-)$ are the standard invariants in the Malle conjecture \cite[Section 2]{Malle}.
\subsection{Residue map} Let~$G$ be a non-trivial finite \' etale $F$-group scheme. Let~$v$ be a finite place of~$F$ such that $\Gamma_v$-group $G(\oF)=G(\overline{F_v})$ is unramified (equivalently, the $\Fv$-group scheme~$G_{\Fv}$ is the base change of an \' etale $\Ov$-group scheme) and such that $G$ is tame at~$v$ (by this we mean that $\gcd(\#G(\oF),q_v)=1$, where $q_v=~\#\kappa_v$).  We will define a map of pointed sets $\psi_v:H^1(\Gamma_v, G(\oF))\to G_*(\Fv)$. If~$G$ is commutative, the map is a group homomorphism. The maps are used to define heights in Subsection~\ref{secdefheight}.

We denote by~$\Fv^{\tame}$ 
the maximal tame extension 
contained in $\overline{F_v}$ and by $\Gamma_v^{\tame}$ its Galois group over~$F_v$. 
Let~$T_v$ be the ramification group $$T_v:=\Gal(\overline{F_v}/\Fv^{\tame})=\ker(\Gamma_v\to\Gamma_v^{\tame})$$ and $I_v$ the inertia group $$I_v:=\Gal(\overline{\Fv}/\Fv^{\un})=\ker(\Gamma_v\to\Gamma_v^{\un}).$$ Let $I_v^{\tame}$ be the tame inertia group, that is $$I_v^{\tame}:=\Gal(F_v^{\tame}/F_v^{\un})=\coker(T_v\to I_v).$$
The sets 
$H^1(\Gamma_v, G(\oF))$ are finite (by \cite[Theorem 7.5.14]{CohomologyNF}, the group~$\Gamma_v$ is a topologically finitely generated group, and thus by \cite[Chapter III, \S 4, Proposition 9]{Cohomologiegalois}, the pointed set $H^1(\Gamma_v, G(\oF))$ is finite).
\subsubsection{} We will prove and recall several auxiliary results. As~$\Gfs$ is unramified at~$v$, the group~$I_v$, and hence the group~$T_v,$ act trivially on it. We deduce, in particular, that~$\Gamma_v^{\tame}$ acts  on~$\Gfs$. We let~$p_v$ be the characteristic of the residue field~$\kappa_v$.
\begin{lem}
The inflation map $H^1(\Gamma_v^{\tame}, \Gfs)\to H^1(\Gamma_v, \Gfs)$ is a bijection of pointed sets. If $G$ is commutative, then the bijection is an isomorphism of abelian groups.
\end{lem}
\begin{proof}We denote by~$p_v$ the characteristic of~$\kappa_v$. By using that $\gcd(\# \Gfs,p_v)=1$ and the fact that 
the group $T_v$ is a pro-$p_v$-group \cite[Proposition 7.5.1]{CohomologyNF}, we obtain that $\Hom_{\cont}(T_v,\Gfs)=\{1\}.$ 
We deduce that one has $H^1(T_v, \Gfs)=~\{1\}.$ 
Now, the inflation-restriction exact sequence \cite[Chapter I, \S 5.8 (a)]{Cohomologiegalois}, applied to the inclusion $T_v\subset~\Gamma_v,$ reads
$$1\to H^1(\Gamma_v^{\tame}, \Gfs)\to H^1(\Gamma_v, \Gfs)\to 1.$$ If $G$ is commutative, then by \cite[Chapter I, \S 2.6 (b)]{Cohomologiegalois}, the exact sequence is an exact sequence of abelian groups and the bijection is an isomorphism of abelian groups.
\end{proof}
The conjugation in~$\Gamma_v^{\tame}$ defines an action on its abelian normal subgroup~$I_v^{\tame}$ of the group~$\Gamma_v^{\un}=\Gamma_v^{\tame}/I_v^{\tame}$. If we denote by~$\sigma$ the canonical topological generator of~$\Gamma_v^{\un}$, by \cite[Page 410]{CohomologyNF}, the action is determined by $\sigma\cdot \gamma=~\gamma^{q_v},$ where $\gamma\in\Gamma_v^{\un}$. For $k\geq 1$ with $\gcd(k, q_v)=1$, there is a canonical identification of the group of $k$-th roots of unity~$\muu_k$ in~$\overline {F_v}$, with the Galois group of the unique tame totally ramified extension of degree~$k$ of~$F_v^{\un}$ (it is obtained by mapping an automorphism $\gamma$ to $\gamma(\piv^{1/k})/\piv^{1/k}$ for a choice $\piv^{1/k}$ of $k$-th root of an uniformizer~$\piv$, the result is independent of the choices). The group~$\Gamma_v^{\un}$ acts on~$\muu_k$ by $\sigma\cdot \xi=\xi^{q_v}$ and it follows that the canonical homomorphism $I_v^{\tame}\to \muu_k$ is $\Gamma_v^{\un}$-equivariant. 
  
\begin{lem}\label{indmue} The canonical map of $\Gamma_v^{\un}$-groups and $G(\oF)$-groups
\begin{equation}
\label{oglic}
\Hom(\muu_e, G(\oF))\to\Hom (I_v^{\tame}, G(\oF))
\end{equation} 
is a bijection. The induced map of $\Gamma_v^{\un}$-pointed sets $$G_*(\oF)=\Conj(\muu_e, G(\oF))\to\Conj(I_v^{\tame}, G(\oF))=H^1(I_v^{\tame}, G(\oF))$$ is a bijection. 
If~$G$ is commutative, then the bijections are isomorphisms of abelian groups.
\end{lem}
\begin{proof}
The canonical homomorphism $I_v^{\tame}\to\muu_e$ is surjective, hence, the map (\ref{oglic}) is injective. Let us prove that a homomorphism $f:I_v^{\tame}\to G(\overline F)$ factorizes through the surjection $I_v^{\tame}\to\muu_e$.  By \cite[Page 410]{CohomologyNF}, the group~$I_v^{\tame}$ is isomorphic to the group $\prod_{\ell\neq p_v}\ZZ_{\ell}$. In particular, for every~$n$ coprime to~$p_v$, there exists up to isomorphism a unique quotient of order~$n$ of~$I_v^{\tame}$ which is moreover a cyclic group (this fact is equivalent to its more known ``dual" claim, that is, for every~$n$ coprime to~$p_v$ there exists a unique subgroup of $\oplus_{\ell\neq p_v}\QQ_{\ell}/\ZZ_{\ell}\subset\QQ/\ZZ$ of order~$n$ and it is moreover cyclic). 
The homomorphism~$f$ can thus be factorized as $$I_v^{\tame} \to \muu_{\#\Imm(f)}\xrightarrow{\sim}\Imm(f)\hookrightarrow G(\overline F).$$ 
Now, as $\#\Imm(f)|e$, the surjection $I_v^{\tame}\to~\muu_{\#\Imm(f)}$ factorizes as $I_v^{\tame}\to\muu_e\to\muu_{\#\Imm(f)}$ and we deduce the factorization of~$f$ $$I_v^{\tame}\to\muu_{e}\to\muu_{\#\Imm(f)}\xrightarrow{\sim}\Imm(f)\hookrightarrow G(\overline F).$$
We deduce that the first map in the statement, and hence the second map in the statement, are bijections. If~$G$ is commutative, then both maps are isomorphisms of abelian groups. 
\end{proof}
\begin{lem} \label{funddefh} 
The sequence of pointed sets $$0\to H^1(\Gamma_v^{\un}, \Gfs)\to H^1(\Gamma_v, \Gfs)\xrightarrow{\psi^G_v} G_*(F_v)$$ is exact, where the map $H^1(\Gamma_v^{\un}, \Gfs)\to H^1(\Gamma_v, \Gfs)$ is the inflation map (hence, by  \cite[Chapter I, Section 5.8 (a)]{Cohomologiegalois} is injective) and the map $$\psi^G_v:H^1(\Gamma_v, \Gfs)\to G_*(F_v)$$ is the restriction map $$H^1(\Gamma_v, \Gfs)=H^1(\Gamma_v^{\tame}, \Gfs)\to H^1(I_v^{\tame}, \Gfs)^{\Gamma_v^{\un}}=G_*(\Fv).$$ 
Suppose that~$G$ is commutative. Then the sequence is an exact sequence of abelian groups and the map $H^1(\Gamma_v, G(\overline F))\to G_*(F_v)$ is surjective.
\end{lem}
\begin{proof}
The sequence is the inflation-restriction exact sequence \cite[Chapter I, \S 5.8 (a)]{Cohomologiegalois}, applied to the inclusion $\Gamma_v^{\un}\subset~\Gamma_v$ and the first claim follows. When~$G$ is commutative, 
by \cite[Chapter I, \S 2.6 (b)]{Cohomologiegalois}, the sequence is an exact sequence of abelian groups and there is an exact sequence $$H^1(\Gamma_v, G(\overline F))=H^1(\Gamma_v^{\tame}, G(\overline F))\to G_*(F_v)\to H^2(\Gamma_v^{\un}, G(\overline F)).$$ By \cite[Page 387]{CohomologyNF}, one has that $H^2(\Gamma_v^{\un}, G(\overline F))=0.$ The second claim follows. 
\end{proof}
\begin{mydef}
The map $\psi^G_v:H^1(\Gamma_v, G(\oF))\to G_*(\oF)$ provided by Lemma \ref{funddefh} will be called the residue map (at the place~$v$).
\end{mydef}
We have the following version of compatibility with closed immersions.
\begin{lem}\label{residueandsubgroup}
Let~$\iota:R\hookrightarrow G$ be a closed subgroup (in particular, $R(\oF)$ is unramified and tame at~$v$). The diagram
\[\begin{tikzcd}
	{H^1(\Gamma^{\tame}_v,R(\oF))} & {H^1(\Gamma^{\tame}_v, G(\oF))} \\
	{R_*(\oF)} & {G_*(\oF)}
	\arrow[from=1-1, to=1-2]
	\arrow["{\psi^R_v}"', from=1-1, to=2-1]
	\arrow["{\iota_*}"', from=2-1, to=2-2]
	\arrow["{\psi^G_v}", from=1-2, to=2-2]
\end{tikzcd}\]
is commutative.
\end{lem}
\begin{proof}
The diagram 
\[\begin{tikzcd}
	{Z^1(\Gamma_v^{\tame}, R(\oF))} & {Z^1(\Gamma_v^{\tame},G(\oF))} \\
	{\Hom(I_v^{\tame}, R(\oF))} & {\Hom(I_v^{\tame}, G(\oF))}
	\arrow[from=2-1, to=2-2]
	\arrow[from=1-2, to=2-2]
	\arrow[from=1-1, to=1-2]
	\arrow[from=1-1, to=2-1]
\end{tikzcd}\]
is commutative, hence is such the diagram
\[\begin{tikzcd}
	{H^1(\Gamma^{\tame}_v,R(\overline F))} & {H^1(\Gamma^{\tame}_v, G(\overline F))} \\
	{H^1(I_v^{\tame}, R(\oF))^{\Gamma_v^{\un}}=R_*( F_v)} & {H^1(I_v^{\tame}, G(\oF))^{\Gamma_v^{\un}}=G_*(F_v).}
	\arrow["{}"', from=2-1, to=2-2]
	\arrow["{\psi^G_v}", from=1-2, to=2-2]
	\arrow[from=1-1, to=1-2]
	\arrow["{\psi^R_v}"', from=1-1, to=2-1]
\end{tikzcd}\]
The diagram of the lemma is commutative, because it is the composite of the latter diagram and the commutative diagram
\[\begin{tikzcd}
	{R_*(F_v)} & {G_*(F_v)} \\
	{R_*(\overline F)} & {G_*(\overline F).}
	\arrow[from=2-1, to=2-2]
	\arrow[from=1-2, to=2-2]
	\arrow[from=1-1, to=1-2]
	\arrow[from=1-1, to=2-1]
\end{tikzcd}\]
\end{proof}
We are going to prove a compatibility property with twists. 
Let $\sigma\in Z^1(\Gamma_v^{\tame}, G(\oF))$. 
\begin{lem}\label{residueandtwist}
Suppose that $[\sigma]\in H^1(\Gamma_v^{\un}, G(\oF))$. The diagram
\[\begin{tikzcd}
	{H^1(\Gamma^{\tame}_v, \sG(\oF))} && {H^1(\Gamma^{\tame}_v, G(\oF))} \\
	& {G_*(\oF)}
	\arrow["{\lambda_{\sigma}}", from=1-1, to=1-3]
	\arrow["{\psi^{G^{\sigma}}_v}"', from=1-1, to=2-2]
	\arrow["{\psi^G_v}", from=1-3, to=2-2]
\end{tikzcd}\]
is commutative.
\end{lem}
\begin{proof}
The condition $[\sigma]\in H^1(\Gamma_v^{\un}, G(\oF))$ is by Lemma \ref{funddefh} equivalent to the condition that the image of $\sigma|_{I_v^{\tame}}\in \Hom(I_v^{\tame}, G(\oF))^{\Gamma_v^{\un}}$ in $G_*(\Fv)=\Conj(I_v^{\tame}, G(\oF))^{\Gamma_v^{\un}}$ is the trivial element, or in other words, $\sigma|_{I_v^{\tame}}$ is conjugate to the trivial map $I_v^{\tame}\to G(\oF)$. But this means that $\sigma|_{I_v^{\tame}}$ itself is the trivial map. Let now $\widetilde\alpha\in Z^1(\Gamma_v^{\tame}, \sG(\oF))$ be a lift of $\alpha\in H^1(\Gamma_v^{\tame}, \sG(\oF)).$ We have that \begin{align*}\psi_v^G(\lambda_{\sigma}(\alpha))=\psi_v^G([\widetilde\alpha\cdot\sigma])=(\widetilde\alpha\cdot\sigma)|_{I_v^{\tame}}&=\widetilde\alpha|_{I_v^{\tame}}\cdot\sigma|_{I_v^{\tame}}\\&=\widetilde\alpha|_{I_v^{\tame}}\\&=\psi^{\sG}_v(\alpha)
\end{align*}and the claim follows. 
\end{proof}
\subsection{Heights}\label{secdefheight} In this subsection we will define heights and provide examples.
\subsubsection{} Let~$G$ be a non-trivial finite \' etale tame $F$-group scheme.  
For every place~$v$ such that~$G(\oF)$ is unramified at~$v$ and such that $G(\oF)$ is tame at~$v$, we let $\psi^G_v:H^1(\Gamma_v, G(\oF))\to G_*(\Fv)\subset G_*(\oF)$ be the residue map. 
\begin{lem}\label{klasikatreba} Let $x\in H^1(\Gamma_F, G(\oF))$. For~$v\in M_F$, let us denote by~$x_v$ the image of$~x$ in $H^1(\Gamma_v, G(\oF))$. For almost every~$v$ such that~$G(\oF)$ is unramified at~$v$, one has that $$\psi^G_v(x_v)=1\in G_*(\oF).$$
\end{lem}
\begin{proof} By Lemma \ref{funddefh}, we need to verify that for almost all~$v$ such that~$G(\oF)$ is unramified at~$v$, the element~$x_v$ lies in the image of the map $H^1(\Gamma_v^{\un}, G(\oF))\to H^1(\Gamma_v, G(\oF))$. There exists a finite Galois extension $K/F$ such that~$x$ lies in the image of the map $H^1(\Gal(K/F), G(\oF))\to H^1(\Gamma_F, G(\oF)).$ Let~$v$ be a finite place such that~$K$ is unramified at~$v$ and let~$w$ be a place of~$K$ lying over~$v$. One has a commutative square 
\[\begin{tikzcd}
	& {H^1(\Gal(K/F), G(\oF))} & {H^1(\Gamma_F, G(\oF))} \\
	{} & {H^1(\Gal(K_w/F_v), G(\oF))} & {H^1(\Gamma_v, G(\oF)),}
	\arrow[from=1-2, to=2-2]
	\arrow[from=2-2, to=2-3]
	\arrow[from=1-3, to=2-3]
	\arrow[from=1-2, to=1-3]
\end{tikzcd}\]
hence, the element $x_v$ is in the image of $H^1(\Gal(K_w/F_v), G(\oF))$. But, the map $H^1(\Gal(K_w/F_v), G(\oF))\to H^1(\Gamma_v, G(\oF))$ factorizes through the map $$H^1(\Gal(K_w/F_v), G(\oF))\to H^1(\Gamma_v^{\un}, G(\oF)).$$ The statement follows.
\end{proof}
Let~$\Sigma_G$ be the finite set of all finite places of~$F$ for which~$G$ is ramified at~$v$ or for which $G$ is wild at~$v$ (i.e. not tame). 
\begin{mydef}\label{definitionofheight}Let $c:G_*(\overline F)\to\RR_{\geq 0}$ be a counting function. Let $M_F^{\infty}\cup \Sigma_G\subset\Sigma\subset M_F$ be a finite set of places.  For $v\in\Sigma$, we let $c_v: BG(F_v)\to\RR_{\geq 0}$ be functions and for $v\in M_F-\Sigma$ let us set $$c_v=c\circ \psi^G_v:BG(F_v)\to \RR_{\geq 0}.$$ For $v\in M_F,$ we denote by~$H_v$ the function $$H_v:BG(F_v)\to\RR_{>0}\hspace{1cm} x\mapsto q_v^{c_v(x)}.$$ The function $$H=H((c_v)_v):BG(F)\to ~\RR_{>0}\hspace{1cm}x\mapsto\prod_{\vMF}H_v(x_v),$$(the product is finite by Lemma \ref{klasikatreba}) is called the height function defined by~$(c_v)_v$ (sometimes simply the height).
\end{mydef}
The counting function~$c$ will be said to be the {\it type} of the height~$H$. For a height~$H$, we may write~$\Sigma_H$ for the set~$\Sigma$ from the definition.
\begin{lem}\label{boundba}
Let $c:G_*(\overline F)\to\QQ_{\geq 0}$ be a counting function and let~$H_1$ and~$H_2$ be two heights having~$c$ for its type. There exists $C_1, C_2>0$ such that for every $x\in BG(F)$ one has that $$C_1\leq \frac{H_1(x)}{H_2(x)}\leq C_2.$$
\end{lem}
\begin{proof} There exists a finite set of places~$\Sigma$, such that for all $v\in M_F-\Sigma$, one has that $H_{1v}=H_{2v}$. As the sets~$BG(F_v)$ are finite, it follows that $$\frac{H_1(x)}{H_2(x)}\geq\prod_{v\in \Sigma}\frac{H_{1v}(x_v)}{H_{2v}(x_v)}\geq \min_{x\in\prod_{v\in\Sigma}(BG(F_v))}\bigg\{\frac{\prod_{v\in\Sigma}H_{1v}(x_v)}{\prod_{v\in\Sigma}H_{2v}(x_v)}\bigg\}.$$ Similarly, one establishes that $\frac{H_1}{H_2}$ is bounded above.
\end{proof}
\subsubsection{}\label{pullbackcountfunc} Let~$G$ be a non-trivial finite \' etale tame group scheme over~$F$ and let~$\sigma\in Z^1(\Gamma_F, G(\oF))$. Let $c:G_*(\oF)=(\sG)_*(\oF)\to\RR_{\geq 0}$ be a counting function. 
Let $\iota:R\hookrightarrow \sG$ be a closed subgroup. By Lemma \ref{funcinj}, we have a pointed $\Gamma_F$-equivariant map $\iota_{*}:R_*(\oF)\to G_*(\oF)$ which is of trivial kernel. The map $$\iota^*c:R_*(\oF)\to\RR_{\geq 0}\hspace{1cm}x\mapsto c(\iota_*(x))$$ is thus a counting function.
\begin{lem}\label{composingheights} Let $\Sigma_G\cup\Sigma_{G^{\sigma}}\subset \Sigma\subset M_F^0$ be finite. For~$v\in M_F-\Sigma$, we choose functions $c_v:BG(F_v)\to\RR_{\geq 0}$ and set~$H_v$ and~$H$ as in  as in Definition \ref{definitionofheight}. The function $$BR(F)\xrightarrow{\iota} B(\sG)(F)\xrightarrow{\lambda_{\sigma}} BG(F)\xrightarrow{H}\RR_{>0}$$is the height $H((c_v\circ\iota)_v)$ (hence, having~$\iota^*c$ for its type).
\end{lem}
\begin{proof}
Let~$v\in M_F^0-(\Sigma_H\cup\Sigma_G\cup\Sigma_{G^{\sigma}})$ be such that $[\sigma]\in H^1(\Gamma_v^{\un}, G(\oF))$. Let~$x\in BR(F)$. By Lemma \ref{residueandtwist} and Lemma \ref{residueandsubgroup}, one has that \begin{align*}
H_v(\lambda_\sigma(\iota(x_v)))&=\exp\big(\log(q_v)\cdot {c(\psi^G_v(\lambda_{\sigma}(\iota(x_v))))}\big)\\&=\exp\big(\log(q_v)\cdot {c(\psi_v^{(\sG)}(\iota(x_v)))}\big)\\&=\exp\big(\log(q_v)\cdot c(\iota_*(\psi^R_v(x_v)))\big).
\end{align*}
The statement follows.
\end{proof}
\subsubsection{} In this paragraph we provide an example of a height, which appears in the Malle conjecture.
%

Let~$v$ be a finite place of~$F$. Let~$n$ be a positive integer and let~$G$ be a  non-trivial constant tame $\Fv$-group scheme (we may write~$G$ for $G(\oF)$). We suppose that~$G$ is embedded as a transitive subgroup $\iota:G\hookrightarrow \mathfrak S_n$. The embedding~$\iota$ defines an action of~$G$ on the set $\{1\doots n\}$. One has a canonical surjection $$\Hom_{\cont}(\Gamma_v, G)\to H^1(\Gamma_v, G)
.$$We may use the same notation for the torsor and the underlying algebra. For a $G$-torsor~$x$ over~$F_v$, we let $\widetilde x:\Gamma_v\to G$ be its lift. The composite homomorphism $\iota\circ {\widetilde x}$ defines a $\Gamma_v$-action on the set $\{1\doots n\}$, and hence by Subsection \ref{equivalenceofcat}, up to isomorphism, a degree~$n$ \' etale algebra~$E_x$ over~$F_v$. We set $$\Delta^{\iota}(x):=\Delta(E_x)\subset\Ov$$where~$\Delta$ denotes the discriminant ideal.
\begin{lem}
Let $x\in H^1(\Gamma_v, G(\oF))$. One has that $$\ord_v(\Delta^{\iota}(x))=c_{{\iota}}(\psi_v(x)).$$
\end{lem}
\begin{proof}
Consider the base change $x_{\Fv^{\un}}=x\otimes _{\Fv}\Fv^{\un}$. The $G$-torsor~$x_{\Fv^{\un}}$ over~$\Fv^{\un}$ is given by the continuous homomorphism~$\widetilde x|_{\Gamma_{\Fv^{\un}}}$. By composing~$\widetilde x|_{\Gamma_{\Fv^{\un}}}$ with the inclusion~$\iota$, we obtain a continuous action of~$\Gamma_{\Fv^{\un}}$ on $\{1\doots n\}$, and the resulting $\Fv^{\un}$-algebra~$E_{x\otimes {\Fv^{\un}}}$ is isomorphic to $E_x\otimes_{\Fv}\Fv^{\un}$.  Moreover, the inclusion of rings $x\to x_{\Fv^{\un}}$ is unramified, hence, one has equality of the ideals of~$\Ov$:
$$\Delta(E_{x\otimes\Fv^{\un}})\cap\Ov=\Delta(E_x)=(\piv^{v(\Delta(E_x))})=(\piv^{v(\Delta^{\iota}(x))}).$$ 
Let $\sigma_1\doots \sigma_k$ be the orbits of~$\Gamma_{\Fv^{\un}}$ for its action on $\{1\doots n\}$. The indices of point stabilizers in~$\Gamma_{\Fv^{\un}}$ are, by point stabilizer lemma, equal to the cardinalities $\# \sigma_1, \doots \#\sigma_k$. A point stabilizer~$H_j$ in the orbit~$\sigma_j$ defines a degree $\#\sigma_j$ totally ramified extension $(\overline{F_v})^{H_j}$ of~$\Fv^{\un}$, hence its discriminant ideal is $(\piv^{\#\sigma_j -1})$. One has that $$(E_{x\otimes{\Fv^{\un}}})\cong\prod_{i=1}^k(\overline{F_v})^{H_k},$$ thus the discriminant ideal of~$E_{x\otimes{\Fv^{\un}}}$ is equal to $$(\piv^{\sum (\#\sigma_k-1)})=(\piv^{n-k}).$$ We deduce that $$v(\Delta^{\iota}(x))=n-k.$$ On the other side, the homomorphism $\widetilde x:\Gamma_v\to G$ factorizes through $\widetilde x':\Gamma_v^{\tame}\to G$ and one has that $\widetilde x'|_{I_v^{\tame}}\circ (\Gamma_{\Fv^{\un}}\to I_v^{\tame})=\widetilde x|_{\Gamma_{\Fv^{\un}}}.$ The image of $x\in H^1(\Gamma_v, G)=H^1(\Gamma_v^{\tame}, G)$ in~$G_*$ under~$\psi_v$ is the class of the homomorphism $x_e:\muu_e\to G$ for which $x_e\circ (I_v^{\tame}\to\muu_e)=\widetilde x'|_{I_v^{\tame}}$. Let~$\xi\in F_v^{\un}$ be a primitive $e$-th root of unity. The value of $c_{{\iota}}$ on the class of~$x_e$ is defined to be the index $$\ind(\iota(x_e(\xi))).$$ The cyclic group $\iota(x_e(\muu_e))=\iota (\widetilde x'(I^{\tame}_v)=\iota(\widetilde x(\Gamma_{\Fv^{\un}})$ has~$k$ orbits, thus the generator $\iota(x_e(\xi))$ is a product of~$k$ cycles. We obtain that $$\ord_v(\Delta^{\iota}(x))=n-k=\ind(\iota(x_e(\xi)))=c_{{\iota}}(\psi_v(x)),$$ as claimed.
\end{proof}
\subsubsection{}\label{discofgtorsor} We now establish that for~$G$ non-trivial finite \' etale tame group scheme over~$F$, one has that the function $$BG(F)\to\RR_{>0},\hspace{1cm}x\mapsto N(\Delta(x))$$ is a height. 

Let~$v$ be a place such that $G(\overline F)$ is unramified and tame at~$v$. Let $x\in H^1(\Gamma_v, G(\overline F))$ and let $\widetilde x:\Gamma_F\to G(\overline F)$ be a continuous crossed homomorphism lifting $x$. By the fact that $G(\overline F)$ is unramified at $v$, we deduce that the restriction $\widetilde x|_{I_v^{\tame}}$ is a homomorphism, and that the restriction $\widetilde x|_{I_v^{\tame}}$ does not depend on the choice of the lift of $\widetilde x$. We let $x_e:\muu_e\to G(\overline F)$ be  the homomorphism given by Lemma \ref{indmue}.
\begin{prop}Let $\vMFz$ be such that $G$ is unramified at $v$. For $x\in H^1(\Gamma_v, G(\oF))$, one has that $$\ord_v(\Delta(x))=c_{\Delta}(\psi_v(x)).$$
\end{prop}
\begin{proof}
Let $\widetilde x:\Gamma_v^{\tame}\to G(\overline F)$ be a continuous crossed homomorphism which lifts $x\in H^1(\Gamma_v, G(\overline F))=H^1(\Gamma_v^{\tame}, G(\overline F)).$ As~$G(\overline F)$ is unramified at~$v$, for $\gamma, \gamma'\in I_v^{\tame}$, one has that $$\widetilde x(\gamma\gamma')=\widetilde x(\gamma)(\gamma\cdot \widetilde x(\gamma'))=\widetilde x(\gamma)\widetilde x(\gamma'),$$i.e. $\widetilde x|_{I_v^{\tame}}$ is a homomorphism. The group scheme $G_{\Fv^{\un}}=G\times _{\Fv}\Fv^{\un}$ is constant, and the $G_{\Fv^{\un}}$-torsor~$x_{\Fv^{\un}}$ is represented by the continuous crossed homomorphism $\widetilde x|_{I_v^{\tame}}:I_v^{\tame}\to G(\overline F)$. We let $M:=\widetilde x(I_v^{\tame}).$ Note that the extension $L:=(\Fv^{\tame})^{\ker(\widetilde x|I^{\tame}_v)}/F_v^{\un}$ is a Galois totally ramified degree $\# M$ extension, hence its discriminant ideal equal to $(\piv^{\#M -1}).$ One has that $x_{\Fv^{\un}}\cong L^{[G(\oF):M]},$ thus, the discriminant ideal of $x_{\Fv^{\un}}$ is equal to $(\piv^{(\#M-1)\cdot [G(\oF):M] }).$ The morphism $x_{\Fv^{\un}}\to x$ is unramified, hence the discriminant ideal of~$x$ is equal to $(\piv^{(\# M-1)\cdot [G(\oF):M]}),$ i.e. $\ord_v(\Delta_v(x))=(\#M-1) \cdot [G(\oF):M].$

On the other side the image of $x\in BG(F_v)=H^1(\Gamma_v^{\tame}, G(\oF))$ in $G_*(\Fv)\subset G_*(\oF)$ is the class of the homomorphism~$x_e$. By definition, the value of~$c_{\Delta}$ at the class of~$x_e$ is equal to $$(\# x_e(\muu_e)-1)\cdot [G(\oF): x_e(\muu_e)]=(\# M-1)\cdot [G(\oF):M].$$ The statement is proven.
\end{proof}
\subsubsection{}  We prove that the number of torsors of bounded height is finite. Let~$G$ be an \' etale group scheme and let $c:G_*(\oF)\to\RR_{\geq 0}$ be a counting function.
\begin{prop}\label{estimatenumber}
Let $H$ be a height having $c$ for its type.  There exists $C,m>0$, such that for every $A>0,$ one has that $$\#\{x\in BG(F)|H(x)\leq A\}\leq CA^m.$$
\end{prop}
\begin{proof}
Let $c_{\Delta}:G_*(\overline F)\to\QQ_{\geq 0}$ be the discriminant counting function. Let $M>\min_{1\neq x\in G_*(\overline F)}\frac{c(x)}{c_\Delta(x)}$ be an integer. The function $M\cdot c$ is a counting function and let us set $$H_M:=H((M\cdot c_v)_v)=H((c_v)_v)^M.$$ By Lemma \ref{boundba}, there exists $C'>0$ such that for every $x\in BG(F)$ one has that  $$C'\cdot H_M(x)>N(\Delta(x)).$$ One obtains for~$A>0$ that
\begin{align*}\#\{x\in BG(F)|H(x)\leq A\}&=\#\{x\in BG(F)|H_M(x)\leq A^M\}\\&\leq \frac{\#\{x\in BG(F)|N(\Delta(x))\leq A^M\}}{C'}.
\end{align*}
It suffices, hence, to prove the statement for the norm of discriminant~$N\circ\Delta$. 
For a positive integer $k$, let us denote for $A>0$ by $\mathcal N^k(A)$ the number of extensions $$\mathcal N^k(A):=\#\{L/F|[L:F]=k, N(\Delta(L))<A\}.$$ By the main theorem of \cite{SchmidtNF} for the case when $F$ is a number field and by \cite[Lemma 2.4]{Gmethodant} for the case when $F$ is a function field, for every $k>0$ one has that there exists integer $r(k)>0$ such that $$\mathcal N^k(A)\leq C(k)\cdot A^{r(k)}$$ for some $C(k)>0$. Let $W(n)$ denotes the set of partitions of $n$. For a partition $\mathcal P\in W(n)$ of $n$, let us set  $$d(\mathcal P)[X]=\prod_{k\in\mathcal P}C(k)X^{r(k)}\in \RR[X].$$ To give an $F$-\' etale algebra of degree $n$ up to isomorphism amounts to give a partition $\mathcal P\in W(n)$ and for every $k\in\mathcal P$, an $F$-extension of degree $k$. The discriminant of an $F$-extension is bounded below by $1$, and thus the number of \' etale algebras of degree $n$ up to isomorphism, which are of discriminant at most $A$ (where $A>0$), is bounded above by $$\sum_{\mathcal P\in W(n)}\prod_{k\in\mathcal P}\mathcal N^k(A)\leq\sum_{\mathcal P\in W(n)}d(\mathcal P)(A).$$ In particular, one has that $$\#\{x\in BG(F)|N(\Delta(x))\leq A\}\leq d(\mathcal P)(A).$$The statement follows.
\end{proof}
\begin{rem}
\normalfont
The bounds used in the proof of Proposition~\ref{estimatenumber} are clearly not optimal. In Subsection~\ref{secequid} for~$G$ commutative, we will establish the precise asymptotic behaviour of $\#\{x\in BG(F)|H(x)\leq A\}$ when $A\to\infty$. 
\end{rem}
\subsection{Conjecture}
Let~$G$ be a non-trivial finite \' etale tame group scheme over~$F$ and let $c:G_*(\oF)\to \RR_{\geq 0}$ be a counting function. 
\subsubsection{}\label{parbreakthin} In this paragraph, we will define breaking thin maps. For $\sigma\in Z^1(\Gamma_F, G(\oF))$, let~$\sG$ be the twist of~$G$ by~$\sigma$.  For a closed subgroup $\iota:R\hookrightarrow \sG,$ we have a counting function $\iota^*c:R_*(\oF)\to \RR_{\geq 0}$ defined in Paragraph~\ref{pullbackcountfunc}. We have a canonical map $BR(F)\to B(\sG)(F)$, which when composed with the pointed bijection $\lambda_{\sigma}:B(\sG)(F)\to BG(F)$ gives a canonical map $BR(F)\to BG(F)$.
\begin{lem}\label{finitefibers}
The fibers of the map $BR(F)\to BG(F)$ have cardinality at most~$\# G(\oF)$.
\end{lem}
\begin{proof}
Because $\lambda_{\sigma}:B(\sG)(F)\to BG(F)$ is a bijection, we can suppose that~$\sG=G$. In the commutative diagram
\[\begin{tikzcd}
	{Z^1(\Gamma_F, R(\overline F))} & {Z^1(\Gamma_F, G(\overline F))} \\
	{BR(F)} & {BG(F),}
	\arrow[from=1-1, to=1-2]
	\arrow[from=1-2, to=2-2]
	\arrow[from=1-1, to=2-1]
	\arrow[from=2-1, to=2-2]
\end{tikzcd}\]
the upper horizontal map is injective, the left vertical map is surjective and the right horizontal map has fibers of size at most~$\# G(\oF))$ (the equivalence class of $\gamma\mapsto f(\gamma)\in Z^1(\Gamma_F, G(\oF))$ is given by maps $\gamma\mapsto g^{-1}f(\gamma)(\gamma\cdot g)$ for $g\in G(\oF)$). We deduce that the lower horizontal map has fibers of size at most~$\# G(\oF)$.
\end{proof}
Lemma \ref{finitefibers} gives that if the number of $x\in BR(F)$ such that $(H\circ (BR(F)\to BG(F))(x)\leq B$ is at least $C_1 B^{r_1}\log(B)^{r_2}$ for $B\gg 0$, where $C_1, r_1, r_2$ are positive constants, then the number of $x\in BG(F)$ such that $H(x)\leq  B$ is at least $(C_1/\# G(F)) B^{r_1}\log(B)^{r_2}$ for $B\gg 0$.
\begin{mydef}
The canonical map $$BR(F)\to B(\sG)(F)\xrightarrow{\lambda_{\sigma}}BG(F)$$ is said to be breaking thin, if one has that $$(a(\iota^*c), b(\iota^*c))> (a(c), b(c)).$$
in the lexicographic order.  
\end{mydef}
It is immediate that one always has $a(\iota^*c)\leq a(c)$, hence the condition in the definition is equivalent to the condition $a(\iota^*(c))=a(c)$ and $b(\iota^*c)>b(c)$. Let us note that if~$G$ is commutative, then the canonical map $R_*(\oF)\to (\sG)_*(\oF)=G_*(\oF)$ is an injection by Lemma~\ref{funcinj} and so $b(\iota^*c)\leq b(c)$, hence there are no breaking thin maps $BR(F)\to BG(F)$. 
\subsubsection{}We present a conjecture which predicts the number of $G$-torsors of bounded height. 
\begin{mydef} 
An element $x\in BG(F)$ is said to be secure if and only if it is not in the image of a breaking thin map. 
\end{mydef}
\begin{prop}\label{usefulcriterion}
Let $x\in BG(F)$ and let $\widetilde x\in Z^1(\Gamma_F, G(\oF))$ be a lift of~$x$. The element~$x$ is secure if and only there are no closed subgroups $\iota:R\hookrightarrow \prescript{}{\widetilde x}G$ such that $(a(\iota^*c), b(\iota^*c))> (a(c), b(c)).$ 
\end{prop}
\begin{proof}
If there exists such a subgroup~$R\subset \prescript{}{\widetilde x}G$, then~$x$ is the image of the element~$1\in BR(F)$  for the canonical map $BR(F)\to B(\prescript{}{\widetilde x}G)(F)\xrightarrow{\lambda_{\widetilde x}}BG(F),$ and hence~$x$ is not secure. Suppose now that~$x$ is not secure. Then there exists $\sigma\in Z^1(\Gamma_F, G(\oF)),$  a closed subgroup $\kappa:M\hookrightarrow \prescript{}{\sigma}G$ for which $(a(\kappa^*c),b(\kappa^*c))>(a(c),b(c))$ and an element $\tau\in BM(F)$ such that~$x$ is in the image of~$\tau$ for the usual map. Let~$\widetilde\tau\in Z^1(\Gamma_F, M(\oF))\subset Z^1(\Gamma_F, \prescript{}{\sigma}G(\oF))$ be a lift of~$\tau$. Part~(5) of Lemma~\ref{listproptwist} gives that the canonical map $\prescript{}{\widetilde{\tau}}\kappa:\prescript{}{\widetilde \tau}M\to \prescript{}{\widetilde \tau}(\sG)$ is a closed immersion. By using Parts (1), (3) and (4) of Lemma~\ref{listproptwist}, we obtain that the diagram
\[\begin{tikzcd}
	{B(\prescript{}{\widetilde\tau}M)(F)} & {B(\prescript{}{\widetilde\tau}(\prescript{}{\sigma}G))(F)=B(\prescript{}{\Lambda_{\sigma}(\widetilde\tau)}G)(F)} \\
	{BM(F)} & {B(\sG)(F)} \\
	& {} & {BG(F)}
	\arrow[from=1-1, to=1-2]
	\arrow["{\lambda_{\widetilde{\tau}}}"', from=1-1, to=2-1]
	\arrow[from=2-1, to=2-2]
	\arrow["{\lambda_{\widetilde\tau}}", from=1-2, to=2-2]
	\arrow["{\lambda_{\sigma}}"', from=2-2, to=3-3]
	\arrow["{\lambda_{\Lambda_{\sigma}(\widetilde\tau)}}", from=1-2, to=3-3]
\end{tikzcd}\]
is commutative. As the image of $1\in B(\prescript{}{\widetilde\tau}M)(F)$ for the left vertical map is~$\tau$, hence, one has that~$x$ is the image of~$1\in B(\prescript{}{\widetilde\tau}M)(F)$ for the map given by composing $\lambda_{\Lambda_{\sigma}(\widetilde\tau)}$ with the upper horizontal map. Now, by Lemma \ref{twistsandstar}, the canonical map $(\prescript{}{\widetilde\tau}M)_{*}\to (\prescript{}{\widetilde\tau}(\sG)_{*}$ identifies with the map $M_{*}\to (\sG)_{*}=G_{*}$, and thus $ (a(\prescript{}{\widetilde\tau}\kappa^*c),b(\prescript{}{\widetilde\tau}\kappa^*c))>(a(c),b(c)).$ The image of $1\in B(\prescript{}{\Lambda_{\sigma}(\widetilde\tau)}G)(F)$ for the map $\lambda_{\Lambda_{\sigma}(\widetilde\tau)}$ is~$x$, it follows that the cocycles $\Lambda_{\sigma}(\widetilde\tau)$ and~$\widetilde x$ are cohomologous. We let~$k\in G(\oF)$ be such that $(\Lambda_{\sigma}(\widetilde\tau))(\gamma)=(k^{-1})(\widetilde x(k))(\gamma(k))$ for every~$\gamma\in \Gamma_F$. We let~$\iota:R\hookrightarrow \prescript{}{\widetilde x}G$ be the closed subscheme corresponding to the $\Gamma_F$-group which is the image of~$\prescript{}{\widetilde\tau}M(\oF)$ for the map $\prescript{}{\Lambda_{\sigma}(\widetilde\tau)}G(\oF)\to\prescript{}{\widetilde x}G(\oF)$ given by $g\mapsto kgk^{-1}.$ It follows from Lemma \ref{twistsandstar} and from above that $$(a(\iota^*c),b(\iota^*c))=(a(\prescript{}{\widetilde\tau}\kappa^*c),b(\prescript{}{\widetilde\tau}\kappa^*c))>(a(c),b(c)).$$The statement is proven. 
\end{proof}
\begin{conj}\label{malletale} 
Let $H$ be a height on $BG(F)$ having~$c$ for its type.
\begin{enumerate}
\item ($F$ is a number field) There exists $A>0$ such that 
\begin{align*}
\sum_{\substack{x\in BG(F)\text{ is secure}\\H(x)\leq B}}\frac{1}{\#\Aut(x)}\sim_{B\to \infty}A\cdot B^{a(c)}\log(B)^{b(c)-1}.
\end{align*}
\item ($F$ is a global field) One has that $$\#\{x\in BG(F)|\text{$x$ is secure, }H(x)\leq B\}\asymp_{B\to \infty} B^{a(c)}\log(B)^{b(c)}.$$
\end{enumerate}
\end{conj}
\begin{rem}
\normalfont
Note that as the cardinalities of the automorphism groups~$\Aut(x)$ for $x\in BG(F)$ are bounded by~$\#G(\oF)$, Conjecture \ref{malletale} for the function field case is equivalent to the version where a point $x\in BG(F)$ is counted with the multiplicity $\#\Aut(x)$.
\end{rem}
\begin{rem}
\normalfont  If $\sigma\in Z^1(\Gamma_F, G(\oF))$ and $y\in B(\sG)(F)$ is secure, then $\lambda_{\sigma}(y)\in BG(F)$ is secure.  Indeed, let $\widetilde {y}\in Z^1(\Gamma_F, G(\oF))$ be a lift of~$y$. Then $\prescript{}{\widetilde y}(\sG)=\prescript{}{\Lambda_{\sigma}(y)}G$ and the claim follows from Proposition \ref{usefulcriterion}. It follows that Conjecture \ref{malletale} is valid for the finite group scheme~$G$ with the respect to the height $H$ if and only if it is valid for the finite group scheme~$\sG$ with the respect to the height $H\circ \lambda_{\sigma}$. 
\end{rem}
\section{Proof for~$G$ commutative}\label{sectioncommutativecase}
Let~$F$ be a global field and let~$G$ be a non-trivial commutative finite \' etale tame $F$-group scheme. Let~$e$ be the exponent of~$G(\oF)$. The goal of this section is to prove Conjecture \ref{malletale} on the number of $G$-torsors of bounded height.
\subsection{Adelic space of $BG$}
We recall useful properties of the ``adelic space" of the stack~$BG,$ that is of the restricted product of local cohomologies with the respect to the unramified cohomologies. 
\subsubsection{}For every $v\in M_F$, we endow the finite groups $BG(F_v):=H^1(\Gamma_v, G(\oF))$ with the discrete topology. Let~$\Sigma_G$ be the set of finite places of~$F$ such that~$G(\oF)$ is unramified at~$v$. For every~$v\in M_F^0-\Sigma_G$, we will write~$BG(\Ov)$ for the subgroup $H^1(\Gamma_v^{\un}, G(\oF))$ (a more precise notation would be $B\mathcal G(\Ov)$, where~$\mathcal G$ is an \' etale model of~$G$ over~$\Ov$, but we think that our choice will not make confusions).  
We set $$BG(\AAA):=\sideset{}{'}\prod_{v\in M_F}BG(F_v)[BG(\Ov)] ,$$ where the restricted product is with the respect to the open and compact subgroups $BG(\Ov)\subset BG(F_v)$ for $v\in M_F-\Sigma_G$ (see e.g. \cite[Chapter VIII, \S 6]{CohomologyNF}). 
As for every~$\vMF$, the groups $BG(F_v)$ are Hausdorff and compact, one sees immediately that the topological group~$BG(\AAA)$ is Hausdorff and locally compact. By Lemma~\ref{klasikatreba}, the diagonal map $BG(F)\to\prod_{\vMF}BG(F_v)$ factorizes through the map $i:BG(F)\to BG(\AAA).$
For $r=1,2$, let us denote by $\Sh^r(G)$ the kernel of the diagonal map $$H^r(\Gamma_F, G(\oF))\to \prod_{v\in M_F}H^r(\Gamma_v,G(\oF)).$$ The finiteness of $\Sh^r(G)$ is proven in \cite[Chapter I, Theorem 4.10]{ArithmeticDuality}.

The following result is a restatement of \cite[Theorem 8.3.20 (ii)]{CohomologyNF} (it is a stronger result than finiteness of $\Sh^1(G)$).
\begin{lem}\label{finlocun}
Let $\Sigma\supset\Sigma_G\cup M_F^{\infty}$ be a finite set of places of~$F$. The group $$i(BG(F))\cap\bigg(\prod_{v\in \Sigma}BG(F_v)\times\prod_{v\in M_F^0-\Sigma}B\GG(\Ov)\bigg)$$ is finite.
\end{lem}
\subsubsection{} In this paragraph, we choose normalization of Haar measures on $i(BG(F))$ and $BG(\AAA)$ and calculate the volume of $BG(\AAA)/i(BG(F))$. The volume, but with different normalizations, has been calculated in \cite[Chapter I, Theorem 5.6]{ArithmeticDuality}.  
If $X$ is a discrete space, let us denote by $\coun_X$ the counting measure on~$X$.
\begin{mydef}
For $v\in M_F$, let us denote by $\mu_v$ the Haar measure on $BG(F_v)$ given by $$\mu_v:=\frac{\coun_{BG(F_v)}}{\# G(F_v)}.$$
\end{mydef}
\begin{rem}
\normalfont
In other words, the measure~$\mu_v$ is the counting measure on $BG(F_v)$ weighted by the inverse of the size of the automorphism group of an element in $BG(F_v)$. 
\end{rem}
For every $v\in M_F^0-\Sigma_G$, by using \cite[Lemma 10.14]{Hararicohom}, we have that $$\mu_v(BG(\Ov))=\frac{\# BG(\Ov)}{\# G(F_v)}=\frac{\# H^1(\Gamma_v^{\un}, G)}{\# H^0(\Gamma_v^{\un}, G)}=\frac{\# H^1(\widehat{\ZZ}, G)}{\# H^0(\widehat{\ZZ}, G)}=1.$$
\begin{mydef}
Let us denote by~$\mu$ the unique Haar measure on the locally compact group $BG(\AAA)$ such that the following condition is satisfied: for every finite set of places~$\Sigma'$ containing the set~$\Sigma_G$ and the set of infinite places~$M_F^\infty$, one has for every choice of open subsets  $U_v\subset BG(F_v)$ for $v\in \Sigma'$, that $$\mu\bigg(\prod_{v\in \Sigma '}U_v\times\prod_{v\in M_F-\Sigma'}B\GG(\Ov)\bigg)=\prod_{v\in\Sigma'}\mu_v(U_v).$$ The existence of a such measure is guaranteed by \cite[Proposition 5.5]{Ramakrishnan}. 
\end{mydef}
Endow $i(BG(F))$ with the measure $\frac{\#\Sh^1(G)}{\#G(F)}\cdot \coun_{i(BG(F))}$.  
This is the pushforward measure of the measure $\frac{\coun_{BG(F)}}{\# G(F)}$, i.e. the counting measure on $BG(F)$, weighted by the inverse of the size of the automorphism group. In the next lemma, we calculate the Tamagawa number of the stack~$BG$ (that is the volume for the quotient measure \cite[Chapter VII, \S 2, \no 7, Proposition 10]{Integrationd} of the quotient group $BG(\AAA)/i(BG(F))$ for certain choices of Haar measures). 
\begin{lem}\label{tamagawabg}
One has that:
\begin{align*}
\tau_{BG}&:=\bigg(\mu/\bigg(\frac{\#\Sh^1(G)\cdot\coun_{i(BG(F))}}{\#G(F)}\bigg)\bigg)\big(BG(\AAA)/(i(BG(F)))\big)\\
&=\frac{\# G^*(F)}{\#\Sh^2(G)}.
\end{align*}
\end{lem}
\begin{proof}
The Haar measure~$\mu'$ from \cite[Page 71]{ArithmeticDuality} on the locally compact abelian group $BG(\AAA)$ is normalized by $$\mu'\bigg(\prod_{v\in M_F^0}H^1(\Gamma_{\kappa_v}, G(\overline{F})^{I_v})\times\prod_{v\in M_F^\infty}\{1\}\bigg)=1.$$
By using that the fact that for every $\vMFz$ one has $$\# H^1(\Gamma_{\kappa_v}, G(F_v)^{I_v})=\# H^0(\Gamma_{\kappa_v}, G(F_v)^{I_v})=\# G(\oF)^{\Gamma_v}=\# G(F_v),$$ we obtain that
\begin{align*}
\mu\bigg(\prod_{v\in M_F^0}H^1(\Gamma_v^{\un}, G(\overline{F}))\times\prod_{v\in M_F^\infty}\{1\}\bigg)=\prod_{\vMFi}\#G(F_v)^{-1},
\end{align*}
and hence that $\mu':\mu=\prod_{\vMFi}\#G(F_v)^{-1}$. 
By \cite[Chapter I, Theorem 5.6]{ArithmeticDuality} and the fact that $\#\Sh^1(G^*)=\#\Sh^2(G)$ from \cite[Chapter I, Theorem 4.20 a)]{ArithmeticDuality}, we have that
\begin{align*}
\tau_{BG}&= \frac{\# G(F)\cdot (\mu'/\coun_{i(BG(F))})\big(BG(\AAA)/(i(BG(F)))\big)}{\big(\prod_{v\in M_F^{\infty}}\# G(F_v)\big)\cdot \#\Sh^1(G)}\\
&=\frac{\# G(F)\cdot \frac{\#\Sh^1(G)\cdot \# G^*(F)}{\#\Sh^1(G^*)\cdot \#G(F)}\cdot\prod_{\vMFi}\# G(F_v)}{\big(\prod_{v\in M_F^{\infty}}\# G(F_v)\big)\cdot \#\Sh^1(G)}\\
&=\frac{\# G^*(F)}{\#\Sh^2(G)}.
\end{align*}
\end{proof}
\subsection{Characters}Let $G^*$ be the Cartier dual of~$G$. That is, it is the \' etale $F$-group scheme corresponding to~$\Gamma_F$-module $\Hom(G(\oF), \muu_e),$ where the action is given by $\gamma\cdot f=\gamma\circ f\circ\gamma^{-1}$. 

\subsubsection{}Let~$A$ be a $\Gamma_v$-module. For every~$i\geq 1$, by \cite[Definition 1.4.3]{brauergrothendieck}, one can define residue map: $$r:H^i(\Gamma_v, A)\to H^{i-1}(\Gamma_v^{\un}, A_*),$$  and its kernel is $H^i(\Gamma_v^{\un}, A)$. Note that~$\psi_v^{A}$ is nothing other than $r:H^1(\Gamma_v, A)\to H^0(\Gamma_v^{\un}, A_*)$. For~$\vMF$, we set $$ \inv_v:\Br(F_v)=H^2(\Gamma_v, \overline{\Fv}^{\times})\to \QQ/\ZZ$$ for the invariant map at the place~$v$. For finite~$v$ such that $v(e)=0$, by \cite[Theorem 1.4.14]{brauergrothendieck}, the map $$H^2(\Gamma_v,\muu_e)\to H^2(\Gamma_v^{\un}, \ZZ/e\ZZ)=\bigg(\frac{1}{e}\ZZ\bigg)/\ZZ\subset\QQ/\ZZ\hspace{1cm}x\mapsto -r(x)$$ identifies with the restriction $\inv_v|H^2(\Gamma_v,\muu_e)\to\QQ/\ZZ$. 
\subsubsection{}For $v\in M_F$, we have Tate pairing 
\begin{equation*}
H^1(\Gamma_v, G(\oF))\times H^1(\Gamma_v, G^*(\oF))\to H^2(\Gamma_v,\overline{\Fv}^{\times})\xrightarrow{\inv_v}\QQ/\ZZ, \end{equation*}\begin{equation*}(x,\chi)\mapsto \inv_v(x\cup \chi),
\end{equation*}
where $\cup$ stands for the cup product. 
%
By \cite[Corollary 2.3]{ArithmeticDuality}, the pairing is perfect, that is, it identifies $H^1(\Fv, G^*)$ canonically with the topological dual $H^1(\Gamma_v, G(\oF))^*$.  Whenever $G$ is unramified at~$v$, the subgroups $H^1(\Gamma_v^{\un}, G(\oF))$ and $H^1(\Gamma_v^{\un}, G^*(\oF))$ are orthogonal complements for the pairing. 
\begin{lem}\label{tatepairing}
Let $v\in M_F^0$ be a place such that $G(\overline F)$ is unramified at $v$. Let $x\in H^1(\Gamma_v, G(\overline F))$ and let $\chi_v\in H^1(\Gamma_v^{\un}, G^*(\overline F))\subset H^1(\Gamma_v, G(\oF)).$ Denote by $\widetilde\chi_v:\Gamma_v^{\un}\to G^*(\oF)$ a crossed homomorphism lifting $\chi$. One has that \begin{multline*}\inv_v(x\cup\chi)=-(\widetilde\chi(\Frob_v)\circ\psi^G_v(x))\\(\in\Hom(\muu_e,\muu_e)=\ZZ/e\ZZ=\bigg(\frac{1}{e}\ZZ\bigg)/\ZZ\subset\QQ/\ZZ).$$ 
\end{multline*}
\end{lem}
\begin{proof}
By \cite[Proposition 6.6]{cohomologicalinvariants}, one has that $$\inv_v(x\cup\chi)=-r(x\cup \chi)=-(r(x)\cup \chi),$$
where the cup product on the right hand side $$H^0(\Gamma_v^{\un}, G_*(\oF))\times H^1(\Gamma_v^{\un}, G^*(\oF)) \to H^1(\Gamma_v^{\un}, \ZZ/e\ZZ)=\ZZ/e\ZZ$$ is defined by the pairing 
$$G_*(\oF)\times G^*(\oF)\to \Hom(\muu_e,\muu_e)=\ZZ/e\ZZ\hspace{1cm}(f,g)\mapsto g\circ f.$$ We have that $r(x)=\psi^G_v(x)\in G_*(\Fv)\subset G_*(\oF)$. For the identification $$Z^1(\Gamma_v^{\un}, G^*(\oF))\xrightarrow{\sim} G^*(\oF)\hspace{1cm} y\mapsto y(\Frob_v),$$ by \cite[Chapter XIII, Proposition 1]{corpslocaux}, one has that~$\chi\in H^1(\Gamma_v^{\un}, G^*(\oF))$ is the image of~$\widetilde\chi (\Frob_v)$ under the quotient map. We deduce that 
$$r(x)\cup\chi=(\widetilde\chi (\Frob_v))\circ\psi^G_v(x).$$ The statement follows.
%
%
\end{proof}
Using the injective homomorphism $$(\QQ/\ZZ)\to S^1\hspace{1cm}q+\ZZ\mapsto \exp(2i\pi q),$$ we may identify $\QQ/\ZZ$ with a subgroup of~$S^1$. 
For $x\in H^1(\Gamma_v, G(\oF))$ and $\chi\in H^1(\Gamma_v, G^*(\oF))$, we write~$$\chi(x):=-\inv_v(x\cup\chi)\in S^1.$$
\subsubsection{}One has a global pairing $$BG(\AAA)\times BG^*(\AAA)\to S^1,\hspace{1cm}((x_v)_v,(\chi_v)_v)\mapsto \prod_{v}\chi_v(x_v).$$ By \cite[Page 56]{ArithmeticDuality}, the pairing is perfect. By \cite[Chapter I, Theorem 4.10]{ArithmeticDuality}, the image of $i:BG(F)\to BG(\AAA)$ is discrete, closed and cocompact. Moreover, the same theorem gives that the subgroups $i(BG(F))\subset BG(\AAA)$ and $i(BG^*(F))\subset BG^*(\AAA)$ are orthogonal complements.
\begin{lem}\label{smallemch}
Let $\Sigma '\supset M_F^\infty\cup \Sigma_G$ be a finite set of places of $F$. For $v\in\Sigma'$, let $T_v\subset BG(F_v)$ be a subgroup. One has that $$\bigg(\prod_{v\in \Sigma'}T_v\times\prod_{v\in M_F^0-\Sigma'}B\GG(\Ov)\bigg)^{\perp}=\prod_{v\in\Sigma'}T_v^{\perp}\times\prod_{M_F^0-\Sigma'}B\GG^*(\Ov).$$
\end{lem}
\begin{proof}
Let $\chi \in \bigg(\prod_{v\in \Sigma'}BG(F_v)\times\prod_{v\in M_F^0-\Sigma'}B\GG(\Ov)\bigg)^{\perp}$. Let $w$ be a place of $F.$ Suppose that $w\in \Sigma'$. For every $x\in T_w,$ one has that $\chi ((x)_w,(1)_{v\neq w})=\chi_w(x)=1$, hence, $\chi_w\in T_w^\perp$. Suppose that $w\in M_F^0-\Sigma'$. For every $x\in B\GG(\OO_w),$ one has that $\chi((x)_w,(1)_{v\neq w})=1$, hence $\chi_w\in B\GG^*(\OO_w)$. The claim follows.
\end{proof}
\subsection{Representation associated to a character} In this subsection, we are going to ``twist" the representation given by a $\Gamma_F$-invariant subset of $G_*(\oF)$ by a character $\chi\in H^1(\Gamma_F, G^*(\oF))$. We are then going to study the analytic behaviour of corresponding $L$-functions. The results will be used in Subsection \ref{Fouriertransformsa} to understand the analytic behaviour of Fourier transforms of heights. 

\subsubsection{}\label{definofreprho} In this paragraph, we define our representations.

Let us denote by $\CC\cdot G_*(\overline F)$ the $\CC$-vector space $$ \CC\cdot G_*(\overline F):=\{(z_{\phi}\cdot \phi)_{\phi}|\phi\in G_*(\overline F), z_{\phi}\in \CC\}.$$ For $\Gamma_F$-invariant subset $W\subset G_*(\oF),$ we write $\CC\cdot W$ for the subspace of $\CC\cdot G_*(\oF)$ generated by the elements of~$W.$  
For $\eta\in G^*(\oF)$ and $\phi\in G_*(\oF)$ we may write $\eta(\phi)$ for $$\eta(\phi)=\eta\circ\phi\in\Hom(\muu_e,\muu_e)=\ZZ/e\ZZ=\bigg(\frac{1}{e}\ZZ\bigg)/\ZZ\subset\QQ/\ZZ\subset S^1.$$ Let $\widetilde \chi\in Z^1(\Gamma_F, G^*(\overline F)).$ For a $\Gamma_F$-invariant subset $W\subset G^*(\oF)$, we define a map $\rho^W_{\widetilde\chi}: \Gamma_F\to\End(\CC\cdot W)$ by  
\begin{align*}
(\rho^W_{\widetilde\chi}(\gamma))_{\phi, \phi'}=\begin{cases} (\widetilde\chi(\gamma))(\phi)&\quad \text{if } \phi=\gamma\cdot\phi',\\
0&\quad\text{otherwise.}
\end{cases}
\end{align*}
Note that every $\rho^W_{\widetilde\chi}(\gamma)$ is a generalized permutation matrix (i.e. it can be obtained from a permutation matrix by replacing ones by non-zero numbers) having $e$-th roots of unity for its non-zero entries. In particular, one has that $\rho^W_{\widetilde\chi}(\Gamma_F)\subset \GL(\CC\cdot W)$. 
\begin{lem}\label{rhorep} Let $\widetilde\chi\in Z^1(\Gamma_F, G^*(\overline F))$ and let $W\subset \CC\cdot G_*(\oF)$ be a $\Gamma_F$-invariant subset.
\begin{enumerate}
\item  For $\gamma,\gamma'\in\Gamma_F$, one has that $\rho^W_{\widetilde\chi}(\gamma\gamma')=\rho^W_{\widetilde\chi}(\gamma)\rho^W_{\widetilde\chi}(\gamma')$.
\item The homomorphism $\rho^W_{\widetilde\chi}:\Gamma_F\to\GL(\CC\cdot W)$ is continuous. In other words, $\rho^W_{\widetilde\chi}$ is a Galois representation.
\item Let $\theta\in G^*(\oF)$ and let $\gamma\in\Gamma_F$. We denote by $D^{\theta}(\gamma)$ the diagonal matrix such that for every $\phi\in G_*(\oF)$ the entry at the place $(\phi,\phi)$ is $\theta(\phi)$. One has that $$\rho^W_{\widetilde\chi}(\gamma)\circ D^{\theta}(\gamma)=D^{\theta}(\gamma)\circ\rho^W_{(\widetilde\chi)(\gamma\cdot \theta)(\theta^{-1})}(\gamma).$$Thus, the representations $\rho_{\widetilde\chi}$ and $\rho_{(\widetilde\chi)(\gamma\cdot\theta)(\theta^{-1})}$ are isomorphic.
\end{enumerate}
\end{lem}
\begin{proof}
\begin{enumerate}
\item Let $\gamma, \gamma'\in\Gamma_F$ and let $\phi, \phi'\in G_*(\overline F).$ Suppose that $\phi\neq (\gamma\cdot\gamma')\cdot\phi'$. We have that $$(\rho^W_{\widetilde\chi}(\gamma)\cdot\rho^W_{\widetilde\chi}(\gamma'))_{\phi, \phi'}=\rho^W_{\widetilde\chi}(\gamma)_{\phi',\gamma'\cdot\phi}\cdot \rho^W_{\widetilde\chi}(\gamma')_{\gamma'\cdot\phi,\phi}=0=(\rho^W_\chi(\gamma\gamma'))_{\phi, \phi'}.$$
Suppose that $\phi=\gamma\cdot(\gamma'\cdot\phi')$. By using that $\gamma^{-1}\cdot\phi=\gamma^{-1}\circ\phi\circ\gamma$ and by using that $\Gamma_F$ acts trivially on $\Hom(\mu_e,\mu_e)$, we obtain that
\begin{align*}
(\rho^W_{\wchi}(\gamma\cdot\gamma'))_{\phi, \phi'}&=\big(\widetilde \chi(\gamma\gamma')\big)(\phi)\\
&=\big(\widetilde\chi(\gamma)\cdot\big(\gamma\cdot \widetilde\chi(\gamma')\big)\big)(\phi)\\
&=\big(\widetilde\chi(\gamma)\cdot\big(\gamma\circ(\widetilde\chi(\gamma'))\circ\gamma^{-1}\big)\big)(\phi)\\
&=\big((\widetilde\chi(\gamma))(\phi)\big)\cdot \big(\gamma\circ(\widetilde\chi(\gamma'))\circ\gamma^{-1}\circ\phi\big)\\
&=\big((\widetilde\chi(\gamma))(\phi)\big)\cdot (\gamma\circ\widetilde\chi(\gamma')\circ(\gamma^{-1}\cdot\phi)\circ\gamma^{-1})\\
&=\big((\widetilde\chi(\gamma))(\phi)\big)\cdot (\widetilde\chi(\gamma') (\gamma^{-1}\cdot\phi))\\
&=(\rho^W_{\widetilde\chi}(\gamma))_{\phi,\gamma^{-1}\cdot\phi}\cdot(\rho^W_{\widetilde\chi}(\gamma'))_{\gamma^{-1}\cdot\phi, (\gamma')^{-1}\cdot\gamma^{-1}\cdot \phi}\\
&=\big(\rho^W_{\wchi}(\gamma)\cdot\rho^W_{\wchi}(\gamma')\big)_{\phi, \phi'}.
\end{align*}
\item The kernel of the crossed homomorphism $\widetilde\chi:\Gamma_F\to G^*(\oF)$ is an open normal subgroup of finite index. 
One has that $\widetilde\chi$ is $\ker (\widetilde\chi)$-invariant. 
 There exists a normal subgroup~$N\subset\Gamma_F$ of finite index such that for $\gamma,\gamma'\in\Gamma_F$ having the same image in $\Gamma_F/N$ one has that $\gamma\cdot \phi=\gamma'\cdot \phi$ for every $\phi\in W$.  Now, it follows from the definition of $\rho^W_{\wchi}$ that for $\gamma,\gamma'\in \Gamma_F$ with the same image in $\Gamma_F/( N\cap\ker (\widetilde\chi))$, one has that $\rho^W_{\wchi}(\gamma)=\rho^W_{\wchi}(\gamma')$, i.e. that $\rho^W_{\wchi}$ factorizes through $\Gamma_F/( N\cap\ker (\widetilde\chi))$. The subgroup $N\cap\ker (\widetilde\chi))\subset\Gamma_F$ is an open normal subgroup of finite index. We deduce that $\rho^W_{\wchi}$ is continuous.
\item Let $\phi,\phi'\in~W$. By using that the canonical action of $\Gamma_F$ on $\Hom(\muu_e,\muu_e)$ is trivial and that $\rho^W_{\wchi}(\gamma)_{\phi, \phi'}=0$ unless $\phi'=\gamma^{-1}\cdot\phi$, we obtain that
\begin{align*}
(D^{\theta}(\gamma)^{-1}\circ \rho^W_{\widetilde\chi}(\gamma)\circ D^{\theta}(\gamma))_{\phi,\phi'}&=(\rho^W_{\widetilde\chi}(\gamma))_{\phi,\phi'}\cdot\frac{\theta(\phi')}{\theta (\phi)}\\
&=(\rho^W_{\widetilde\chi}(\gamma))_{\phi,\phi'}\cdot \frac{\theta\circ\gamma^{-1}\circ\phi\circ\gamma}{\theta(\phi)}\\
&=(\rho^W_{\widetilde\chi}(\gamma))_{\phi,\phi'}\cdot\frac{\gamma\circ\theta\circ\gamma^{-1}\circ\phi}{\theta(\phi)}\\
&=(\rho^W_{\widetilde\chi}(\gamma))_{\phi,\phi'}\cdot\frac{(\gamma\cdot \theta)(\phi)}{\theta(\phi)}\\
&=(\rho^W_{(\widetilde\chi)(\gamma\cdot\theta)(\theta^{-1})}(\gamma))_{\phi,\phi'}.
\end{align*}
The claim follows.
\end{enumerate}
\end{proof}
For $\chi\in H^1(\Gamma_F, G^*(\oF))$, we will denote by~$\rho^W_{\chi}$ the (isomorphism class of) representation~$\rho^W_{\widetilde\chi}$ for a lift~$\widetilde\chi$ of~$\chi$. In the rest of the subsection, the choice of a lift may be implicit. As similar matrices have same characteristic polynomials, the notations such as $\tr(\rho^W_{\chi}(\Frob_v))$ or $\tr(\rho^W_{\wchi}(\Frob_v))$ are unambiguous. 

%
\begin{lem}\label{traceofl} Let $\chi\in H^1(\Gamma_F, G^*(\oF))$.  Let $v\in M_F^0-\Sigma_G$ be such that the character~$\chi$ is unramified at~$v$ (i.e. such that~$\chi_v|_{H^1(\Gamma_v^{\un}, G(\oF))}=1$). One has that 
$$\tr(\rho^W_{\chi}(\Frob_v))=\frac{1}{\# G(F_v)}\sum_{y\in (\psi_v^G)^{-1}(W^{\Frob_v})}\chi_v(y).$$
\end{lem}
\begin{proof}
For every $\phi\in W$, one has that $$(\rho^W_{\widetilde\chi}(\Frob_v))_{\phi, \phi}=\begin{cases} (\widetilde\chi(\Frob_v))(\phi)&\quad \text{if } \phi=\Frob_v\cdot\phi,\\
0&\quad\text{otherwise.}
\end{cases}$$
Lemma \ref{tatepairing} gives that for every $y\in H^1(F_v, G(\oF))$ such that $\psi^G_v(y)=~\phi$, one has that $\widetilde\chi_v(\Frob_v)(\phi)=\chi_v(y),$ where $\widetilde\chi_v$ is the composite morphism $\Gamma_v\to \Gamma_F\xrightarrow{\widetilde{\chi}}G^*(\oF)$.  We deduce that  \begin{align*}\tr(\rho^W_{\chi}(\Frob_v))=\tr(\rho^W_{\widetilde\chi}(\Frob_v))&=\sum _{\phi\in W^{\Frob_v}}\widetilde\chi(\Frob_v)(\phi)\\
&=\sum _{\phi\in W^{\Frob_v}}\widetilde\chi_v(\Frob_v)(\phi)\\
&=\sum_{y\in(\psi^G_v)^{-1}(\phi)}\frac{1}{\#(\psi_v^G)^{-1}(\phi)}\sum_{\hspace{0.2cm}\phi\in W^{\Frob_v}}\chi_v(y).
\end{align*}
By Lemma \ref{funddefh} for every $\phi\in W$, one has that \begin{equation*}\#(\psi_v^G)^{-1}(\phi)=\#H^1(\Gamma^{\un}_v, G(\oF))=\# G(F_v).\end{equation*}
The claim follows.
\end{proof}
\subsubsection{}In this paragraph, we will study $L$-functions of the representations~$\rho^W_{\chi}$.  Let $W\subset G_*(\oF)$ be a $\Gamma_F$-invariant subset. For a lift $\wchi:\Gamma_F\to G^*(\oF)$ of $\chi\in H^1(\Gamma_F, G^*(\oF))$ and a place $v\in M_F^0$, we set $$L_v\big(s,\rho_{\chi}^{W}\big)=\det\big(1-q_v^{-s}\rho_{\wchi}^{W}(\Frob_v)|(\CC\cdot W)^{I_v}\big)^{-1}.$$ %
By Part (3) of Lemma \ref{tatepairing}, one has that $L_v(s, \rho^W_{\chi})$ does not depend on the choice of $\widetilde\chi$, as notation suggests.  The factors $L_v(s,\rho_{\chi}^W)$ are the local Artin $L$-factors of a Galois representation~$\rho^W_{\chi},$ and hence define non-vanishing holomorphic functions in the domain $\Re(s)\geq 1$. In the next lemma, we summarize the properties of a global $L$-function. We recall that in the function field case, we write~$q$ for the cardinality of the field of constants of~$F$.
\begin{lem}\label{lfunf}
Let~$\Sigma\subset M_F$ be finite. 
\begin{enumerate}
\item The product $$L^{\Sigma}(s,\rho_{\chi}^W):=\prod_{v\in M_F^0-\Sigma}L_v(s,\rho_{\chi}^{W})$$ converges normally on compacts in the domain $\Re(s)>1$ and the function $s\mapsto L^\Sigma(s,\rho^W_{\chi})$ is holomorphic in this domain. Moreover, $s\mapsto L^\Sigma(s,\rho^W_{\chi})$ extends to a meromorphic function on~$\CC$ with a possible pole at $s=1$ which is of order at most $\#(W/\Gamma_F)$.
\item If $\chi=1$ is the trivial character, then the pole of $s\mapsto L^{\Sigma}(s,\rho_{\chi}^W)$ at $s=1$ exists and its order is precisely $\#(W/\Gamma_F)$. 
\item If~$F$ is a number field, then $s\mapsto L^\Sigma(s,\rho^W_{\chi})$ is holomorphic in the domain $\{\Re(s)\geq 1\}-\{1\}$. If~$F$ is a function field, then $s\mapsto L^{\Sigma}(s, \rho^W_{\chi})$ is holomorhic in the domain $$\{\Re(s)>0\}-\{1+2ki\pi\log(q)^{-1}|k\in\ZZ\}.$$
\end{enumerate}
\end{lem}
\begin{proof}
\begin{enumerate}
\item As $s\mapsto L_v(s,\rho^W_{\chi})$ are non-vanishing and holomorphic in the domain $\Re(s)>0$, we can suppose that $\Sigma=\emptyset$. The representations~$\rho^W_{\chi}$ are defined over a finite rang $\ZZ$-module (e.g. over $\ZZ[\xi_e],$ where~$\xi_e$ is $e$-th primitive root of~$1$). Now, \cite[Proposition 2.49]{bourquifzh} applies, and the claims on convergence and meromorphic extension follow. Moreover, by the same proposition, the order of the pole at~$s=1$ is equal to the dimension of the subspace $(\CC\cdot W)^{\rho_{\chi}^W}$ fixed by~$\rho^W_{\chi}$. Let~$Q$ be a $\Gamma_F$-orbit in~$W$ and let $\phi\in Q$. For a vector $0\neq (w_{\phi})_{\phi}\in\CC\cdot Q$, which is fixed by~$\rho^W_{\chi}$, one has for every $\gamma\in\Gamma_F$ that $w_{\gamma\cdot\phi}=\widetilde\chi(\gamma)(\phi)\cdot w_{\phi}$. It follows that the dimension of the subspace of $\CC\cdot Q$ fixed by~$\rho^W_{\chi}$ is no more than~$1$. It follows that the subspace $(\CC\cdot W)^{\rho^W_{\chi}}$ fixed by the representation $\rho^{W}_{\chi}$ is of the dimension no more than $\# (W/\Gamma_F)$. We deduce that the order of the pole at $s=1$ is no more that $\#(W/\Gamma_F)$. 
\item Let us now suppose that~$\chi$ is the trivial character, then $\wchi':\gamma\mapsto 1\in G^*(\oF)$ is a lift of~$\chi$. In this case, every generalized permutation matrix $\rho^W_{\wchi'}(\gamma)$ fixes the vector subspace $$\{(x_Q)_{\phi\in Q}|Q\in(W/\Gamma_F), x_Q\in\CC\} \subset \CC\cdot W.$$ Hence, the subspace $(\CC\cdot W)^{\rho^W_{\chi}}$ is of the dimension $\# (W/\Gamma_F)$.  The claim on the order of the pole for the case $\chi=1$ follows.
\item In the case~$F$ is a number field, by Brauer theorem \cite{brauertheorem}, one has that $s\mapsto L(s,\rho^W_{\chi})$ is holomorhic in the domain $\{\Re(s)\geq 1\}-\{1\}$. Suppose that~$F$ is a function field. Let us decompose $\rho^W_{\chi}=\oplus_{i=1}^a\rho_i^{b_i}$ into a direct sum of the irreducible representations, where for $i\neq j$, the representations~$\rho_i$~and $\rho_j$ are not isomorphic and $\rho_1$ is the trivial representation. By using \cite[Page 196]{Milne} and \cite[Chapter VII, \S 6, Theorem 4]{basicnt}, we obtain that \begin{align*}
L(s,\rho^W_{\chi})=\prod_{i=1}^aL(s,\rho_i^{b_i})&=L(s,\rho_1)^{b_1}\cdot\prod_{i=2}^aL(s,\rho_i)^{b_i}\\&=\zeta_F(s)^{b_1}\cdot \prod_{i=2}^aL(s,\rho_i)^{b_i}\\
&=\bigg(\frac{P(q^{-s})}{(1-q^{-s})(1-q^{1-s})}\bigg)^{b_1}\prod_{i=2}^{b_i}L(s,\rho_i)^{b_i},
\end{align*}
where $\zeta_F$ is the Dedekind zeta function of~$F$ and~$P$ is a polynomial such that $P(0)=1$. For every $2\leq i\leq a$, one has by \cite[Pages 196-197]{Milne} that $s\mapsto L(s,\rho_i)$ is holomorphic on~$\CC$. It follows that in the domain $\Re(s)>0$, the set of the poles of $s\mapsto L(s,\rho^W_{\chi})$ is contained in the set $\{1+2ki\pi\log(q)^{-1}|k\in\ZZ\}$, as claimed.
\end{enumerate}
\end{proof}
\subsection{Analysis of height zeta functions}\label{Fouriertransformsa} In this subsection, we will define and analyse the height zeta functions. 
Let $c:G_*(\oF)\to\QQ_{\geq 0}$ be a normalized counting function and let~$H$ be a height having~$c$ for its type. Let $\Sigma_H\supset \Sigma_G\cup M_F^{\infty}$ be the finite set of places such that for~$v\in M_F^0-\Sigma_H,$ one has that $H_v:BG(F_v)\to \RR_{>0}$ is given by$$H_v(x)=q_v^{c(\psi^G_v(x))}.$$ 
The definition of height extends to $H:BG(\AAA)\to \RR_{>0}$ by $$H((x_v)_v)=\prod_{\vMF}H_v(x_v)$$so that for $x\in BG(F)$ one has that $H(x)=H(i(x)),$ where $i:BG(F)\to BG(\AAA)$ is the diagonal map. 
\subsubsection{} The following lemma will be used later.
\begin{lem}\label{codx}
Let $x\in BG(\AAA).$ We define $$\Sigma_x:=\Sigma_H\cup\{v\in M^0_F-\Sigma|i_v(x)\not\in B\GG(\Ov)\}.$$  We set $$C(x):=\prod_{v\in\Sigma_x}\frac{\max_{z\in BG(F_v)}H_v(z)}{\min_{z\in BG(F_v)}H_v(z)}.$$ For every $y\in BG(\AAA)$, one has that $$C(x)^{-1}H(y)\leq H(xy)\leq C(x)H(y).$$
\end{lem}
\begin{proof}
For $v\in M_F-\Sigma_x$, one has that $H_v(x)=1.$ For $v\in\Sigma_x$, we set $$C_v(x):=\frac{\max_{z\in BG(F_v)}H_v(z)}{\min_{z\in BG(F_v)}H_v(z)},$$ so that $C(x)=\prod_{v\in\Sigma_x}C_v(x).$ For every $y\in BG(\AAA)$, one has that \begin{align*}
H(xy)&=\prod_{v\in M_F}H_v(xy)\\
&=\prod_{v\in\Sigma_x}H_v(xy)\cdot\prod_{v\in M_F-\Sigma_x}H_v(y)\\
&\leq \prod_{v\in\Sigma_x}C_v(x)H_v(y)\cdot\prod_{v\in M_F-\Sigma_x}H_v(y)\\
&=C(x)H(y).
\end{align*}
Analogously, one verifies that $C(x)^{-1}H(y)\leq H(x)$.
\end{proof}
For $s\in\CC$ and $\chi_v\in (BG(F_v))^*$,  we define $$\wH_v(s,\chi_v):=\int_{BG(F_v)}H_v^{-s}\overline{\chi_v}\mu_v.$$ 
For every $\chi_v\in BG(F_v)^*,$ the function $s\mapsto\wH_v(s,\chi_v),$ being a sum of finitely many entire functions, is an entire function.

For a character $\chi\in BG(\AAA)^*=BG^*(\AAA)$ we denote by $$ \wH(s,\chi):=\int_{BG(\AAA)}H^{-s}\overline{\chi}\mu.$$
For a place $v,$ the set given by the elements $x\in BG(F_v)$ such that~$H_v$ is invariant for the multiplication by~$x$, forms an open subgroup of the discrete group $BG(F_v)$. We denote it by~$K_{H_v}$. Note that for $v\in M_F^0-\Sigma_H$, one has that $K_{H_v}=B\GG(\Ov).$ We denote by~$K_H$ the compact and open subgroup $$K_{H}:=\prod_{\vMF}K_{H_v}\subset BG(\AAA).$$
We denote $$\Xi_H:=i(BG(F))K_H.$$ 
The following lemma identifies the finite group of characters for which the Fourier transform does not vanish.
\begin{lem}
\label{importantchar}
\begin{enumerate}
\item The subgroup $\Xi_H^{\perp}\subset BG(\AAA)^*$ is finite.
\item The homomorphism $$BG(\AAA)=(BG(\AAA)^*)^*\to (\Xi_H^\perp)^*,\hspace{1cm} x\mapsto x|_{\Xi_H^\perp}$$ is surjective and its kernel identifies is the group $\Xi_H$. In particular, the subgroup $\Xi_H$ is open, closed and of the index $\#\Xi_H^\perp$ in $BG(\AAA)$.
\item There exists a finite set of places~$(\Sigma_G\cup M_F^{\infty})\subset\Sigma$ and a height~$H'$ for which~$H_v'=1$ if $v\in \Sigma$ and $H_v'=q_v^{c_v(\psi_v(x))}$  if $v\in M_F-\Sigma$, such that $$\Xi_{H'}=BG(\AAA).$$
\item Let $\chi \in (BG(\AAA)^*-\Xi_H^{\perp})$ be a character. One has that $\wH(s,\chi)=~0$.
\end{enumerate}
\end{lem} 
\begin{proof}
\begin{enumerate}
\item 
By \cite[Chapter 2, \S 1, \no 7, Corollary 2 of Theorem 4]{TSpectrale} and by Lemma \ref{smallemch}, one has that
\begin{align*}
\Xi_H^\perp&=(i(BG(F))K_H)^\perp\\&=(i(BG(F))^{\perp}\cap K_H^{\perp}\\&=i(BG^*(F))\cap\bigg(\prod_{v\in\Sigma_H}K_{H_v}^{\perp}\times\prod_{v\in M_F^{0}-\Sigma_H}B\GG^*(\Ov)\bigg)
\end{align*}
The finiteness of the later intersection follows from Lemma~\ref{finlocun}.
\item The homomorphism $BG(\AAA)\to (\Xi_H^\perp)^*$ is the dual homomorphism \cite[Chapter II, \S 1, \no 7, Page 124]{TSpectrale} of the inclusion of a finite, and hence closed, subgroup $\Xi_H^\perp\subset BG(\AAA)^*,$ hence by \cite[Chapter II, \S 1, \no 7, Theorem 4]{TSpectrale} is surjective and its kernel is the group $(\Xi_H^\perp)^{\perp}=\overline{\Xi_H}$. But, the subgroup $\Xi_H\subset BG(\AAA)$ is the subgroup generated by a discrete subgroup $i(BG(F))$ and a compact subgroup $K_H$, hence, by \cite[Chapter III, \S 4, \no 1, Corollary 1 of Proposition 1]{TopologieGj}, is closed in $BG(\AAA)$. Thus $\Xi_H=\overline{\Xi_H}=(\Xi_H^\perp)^\perp$ and, hence, $\Xi_H$ is the kernel of $BG(\AAA)\to (\Xi_H^\perp)^*$.
\item By Part (2), we need to find~$H'$ for which $\Xi_{H'}^{\perp}$ is a singleton. There exists a finite $\Sigma_H\cup\Sigma_G\cup\Sigma_{G^*}\cup M_F^{\infty}\subset \Sigma\subset M_F$ such that $$\Xi_H^{\perp}\cap \big(\prod_{v\in\Sigma}BG^*(F_v)\times\prod_{v\in M_F^0-\Sigma}BG^*(\Ov)\big)=\{1\}.$$ Now, let us set~$H'_v=1$ for $v\in\Sigma$ and $H'_v=H_v$ for $v\not\in\Sigma$, so that $K_{H}\subset K_{H'}$ and that $K_{H'}=\prod_{v\in\Sigma}BG^*(F_v)\times\prod_{v\in M_F^0-\Sigma}BG^*(\Ov).$ Thus $$\Xi_{H'}^{\perp}=(i(BG(F))K_{H'})^{\perp}=(i(BG(F))K_HK_{H'})^{\perp}\subset\Xi_H^{\perp}\cap K_{H'}=\{1\}$$and the claim follows. 
\item Let $\vMF$ be such that $\chi_v|_{K_{H_v}}\neq 1$. As $x\mapsto H_v(x)^{-s}$ is $K_{H_v}$-invariant, we deduce that $\wH_v(s,\chi_v)=0$. It follows that $\wH(s,\chi)=0$.
\end{enumerate}
\end{proof}
\subsubsection{} 
We will compare the Fourier transform at the character~$\chi$, with the $L$-function of a representation $\rho^W_{\chi}$. Recall the notation $G_*(\oF)_1=\{x\in G_*(F)| c(x)=1\}.$ 
\begin{lem}\label{ftlfun} Let $\chi\in H^1(\Gamma_F,G^*(\oF))\cap K^{\perp}\subset BG(\AAA)^*$ be a character and let~$\Sigma\subset M_F^0$ be finite. The product 
$$\prod_{\vMFz-\Sigma}\frac{{\wH_v(s,\chi_v)}}{L_v(s,\rho_{\chi}^{G_*(\oF)_1})}$$ converges normally on the compacts in the domain $\Re(s)>~\lambda^{-1}$ to a holomorphic function, where \begin{align*}
1<\lambda:=\begin{cases} \min_{x\in (G_*(F_v)-(\{1\}\cup G_*(F)_1))}c(x), \text{if } \{1\}\cup G_*(F_v)_1\subsetneq G_*(F_v),\\
2\quad\quad\quad\quad\quad\quad\quad\quad\quad\quad\quad\quad\text{, otherwise}.
\end{cases}
\end{align*}
(Here, $G_*(\Fv)_1=(G_*(\oF)_1)^{\Gamma_v}$.)
\end{lem}
\begin{proof}
Let us set $\Sigma_{\chi}$ to be the union of~$\Sigma_H$ and the finite set of places for which~$\chi$ is ramified. By Lemma \ref{traceofl}, for $\vMFz-\Sigma_\chi$ and $s\in\{z:\Re(z)>0\}$, one has that
\begin{align*}
L_v\big(s,\rho^{G_*(\oF)_1}_{\chi}\big)^{-1}&=\det\big(1-q_v^{-s}\rho^{G_*(\oF)_1}_{\chi}(\Frob_v)\big)\\
&=1-q_v^{-s}\Tr\big(\rho^{G_*(\oF)_1}_{\chi}(\Frob_v)\big)+ O(q_v^{-2\Re(s)})\\
&=1-\bigg(\frac{q_v^{-s}}{\#G(F_v)}\sum_{y\in(\psi^G_v)^{-1}(G_*(F_v)_1)}\chi_v(y) \bigg)+O(q_v^{-2\Re(s)}).
\end{align*}
On the other side, for $\vMF^0-\Sigma_{\chi}$ and $s\in\{z:\Re(z)>0\}$, one has that \begin{align*}
\wH_v(s,\chi_v)&=\int_{BG(F_v)}H_v^{-s}\overline{\chi_v}\mu_v\\
&=\frac{1}{\#G(F_v)}\sum_{x\in BG(F_v)} H_v(x)^{-s}\chi_v(x)\\
&=\frac{1}{\# G(F_v)}\bigg(\sum_{x\in B\GG(\Ov)}1 + q_v^{-s}\sum_{x\in (\psi^G_v)^{-1}(G_*(F_v)_1)}\chi_v(x)+\\&\quad\quad\quad\quad\quad+\sum_{x\in BG(F_v)-\big(B\GG(\Ov)\cup(\psi^G_v)^{-1}(G_*(F_v)_1) \big)}H_v(x)^{-s}\chi_v(x)\bigg)\\
&=1+\bigg(\frac{q_v^{-s}}{\#G(F_v)}\sum_{y\in(\psi_v^G)^{-1}(G_*(\Fv)_1)}\chi_v(y)\bigg)+ O(q_v^{-\Re(s)\cdot\lambda}).
\end{align*}
Hence, for $v\in M_F^0-\Sigma_\chi$ and $\Re(s)>0$ one has that $$L_v\big(s,\rho_{\widetilde\chi}^{G_*(\oF)_1}\big)^{-1}\wH_v(s,\chi_v)=1+O(q_v^{-\Re(s)\cdot \lambda}).$$ It follows from \cite[Chapter VII, \S 5, Corollary 1]{basicnt} that the product $$\prod_{v\in (M_F^0-\Sigma)-\Sigma_\chi}L_v(s,\rho_{\wchi}^{G_*(\oF)_1})^{-1}\wH_v(s,\chi_v),$$ 
converges normally on the compacts to a holomorphic function in the domain $\Re(s)>\lambda^{-1}.$ 
Hence the product $$\prod_{v\in M_F^0-\Sigma}L_v(s,\rho_{\chi}^{G_*(\oF)_1})^{-1}\wH_v(s,\chi_v),$$ converges normally on the compacts to a holomorphic function in the domain $\Re(s)>\lambda^{-1}.$ 
\end{proof}

\begin{cor}\label{trivialcharacter}
Let $\chi\in H^1(\Gamma_F, G^*(\oF))\cap K_H^\perp\subset BG^*(\AAA)$ be a character and let $\lambda>1$ be as in Lemma \ref{ftlfun}.  Let $\Sigma\subset M_F$ be finite.
\begin{enumerate}
\item The product defining $\wH_{M_F-\Sigma}(s, \chi):=\prod_{v\in M_F-\Sigma}\wH_v(s,\chi_v)$ converges normally on the compacts in the domain $\Re(s)>1$. The function $s\mapsto\wH_{M_F-\Sigma}(s,\chi)$ extends to a meromorphic function in the domain $\Re(s)> \lambda^{-1}$ with a possible pole at $s=1$ of order at most $b(c)=(G_*(\oF)_1/\Gamma_F)$. In the case~$F$ is a number field, one also has that $s\mapsto\wH_{M_F-\Sigma}(s,\chi)$ is holomorphic in the domain $\{\Re(s)\geq 1\}-\{1\}$. In the case~$F$ is a function field, one also has that $s\mapsto\wH_{M_F-\Sigma}(s,\chi)$ is holomorphic in the domain $\{\Re(s)\geq 1-\epsilon\}-\{1+2ki\pi\log(q)^{-1}|k\in\ZZ\}$ for some $\epsilon>0$.
\item If $\chi=1$ is the trivial character, then the meromorphic function $s\mapsto\wH(s,1)=\wH_{M_F}(s,1)$ admits a pole at $s=1$ of order $b(c)$. Moreover, one has that $$\lim_{s\to 1^+}(s-1)^{b(c)}\wH(s,1)=\lim_{s\to 1^+}(s-1)^{b(c)}\int_{BG(\AAA)}H^{-s}\mu>0.$$
\end{enumerate}
\end{cor}
\begin{proof}
\begin{enumerate}
\item The claims follow by combining Lemma \ref{ftlfun}, Lemma \ref{lfunf} and the fact that the functions  $s\mapsto\wH_v(s,\chi_v)$ are non vanishing for $\Re(s)\geq 0$ if $\chi=1$.
\item From Lemma \ref{ftlfun}, Part (2) of Lemma \ref{lfunf} and the fact the functions  $s\mapsto\wH_v(s,1)$ are non vanishing for $\Re(s)\geq 0$, 
we deduce that $s\mapsto\wH(s,1)$ has a pole at~$s=1$ of order~$b(c)$. The principal value at the pole is equal to \begin{align*}\lim_{s\to 1^+}(s-1)^{b(c)}\wH(s,1)=\lim_{s\to 1^+}(s-1)^{b(c)}\int_{BG(\AAA)}H^{-s}\mu.
\end{align*} and is strictly bigger than zero, as it is nonzero and for $s>1$ one has that $(s-1)^{b(c)}\wH(s,1)>0.$ 
\end{enumerate}
\end{proof}
The following lemma will be used later.
\begin{lem}\label{intfinind}
Let~$V, W\subset BG(\AAA)$ be open subgroups of finite index. Let~$y_1\in V$. One has that $$\lim_{s\to 1^+}(s-1)^{b(c)}\int_{(y_1W)\cap V}H^{-s}\mu>0.$$ 
\end{lem}
\begin{proof}
Let $y_1\doots y_r$ be a set of representatives of classes of~$W$ in~$BG(\AAA)$. For~$s>1$, by Lemma~\ref{codx}, one has that \begin{align*}\int_{BG(\AAA)}H^{-s}\mu&=\sum_{i=1}^k\int_{y_iW}H^{-s}\mu\\
&=\sum_{i=1}^r\int_{y_1W}H(y_iy_1^{-1}t)^{-s}d\mu(t)\\
&\leq\bigg( \sum_{i=1}^rC_i^{-s}\bigg)\cdot\int_{y_1W}H^{-s}\mu
\end{align*}
for certain $C_i>0$. Now, Part~(2) of Corollary~\ref{trivialcharacter} gives that
\begin{align*}
0&<\lim_{s\to 1^+}(s-1)^{b(c)}\int_{BG(\AAA)}H^{-s}\mu\\
&\leq\lim_{s\to 1^+}(s-1)^{b(c)}\bigg(\sum_{i=1}^rC_i^{-s}\bigg)\cdot\int_{y_1W}H^{-s}\mu\\
&=\bigg(\sum_{i=1}^rC_i^{-1}\bigg)\lim_{s\to 1^+}(s-1)^{b(c)}\int_{y_1W}H^{-s}\mu,
\end{align*}
and hence $$\lim_{s\to 1^+}(s-1)^{b(c)}\int_{y_1W}H^{-s}\mu>0.$$
Note that it follows (by setting $W=V$ and $y_1=1$) that $$\lim_{s\to 1^+}(s-1)^{b(c)}\int_{V}H^{-s}\mu>0.$$ 
Let $y_1', y_2'\doots y_k'\in V$ be a set of representatives of classes of $V\cap W$ in~$V$ such that $y_1'=y_1$. For~$s>1$, by Lemma~\ref{codx} we have that:
\begin{align*}
\int_{V}H^{-s}\mu&=\sum_{i=1}^k\int_{ (y_i')(W\cap V)}(H^{-s})\mu\\
&=\sum_{i=1}^k\int_{(y_1')(W\cap V)}(H((y_i')(y_1')^{-1}t)^{-s})d\mu(t)\\
&\leq \bigg(\sum_{i=1}^k(C'_i)^{-s}\bigg)\int_{(y_1)(W\cap V)}H^{-s}\mu\\
\end{align*}
for some~$C'_i>0$. We obtain that
\begin{align*}
0&<\lim_{s\to 1^+}(s-1)^{b(c)}\int_{V}H^{-s}\mu\\
&\leq\lim_{s\to 1^+}(s-1)^{b(c)}\bigg(\sum_{i=1}^k(C_i')^{-s}\bigg)\cdot\int_{(y_1)(W\cap V)}H^{-s}\mu\\
&=\bigg(\sum_{i=1}^k(C_i')^{-1}\bigg)\lim_{s\to 1^+}(s-1)^{b(c)}\int_{(y_1)(W\cap V)}H^{-s}\mu.
\end{align*}
Finally, by remarking that $(y_1W)\cap V\supset (y_1)(W\cap V)$, we deduce that $$\lim_{s\to 1^+}(s-1)^{b(c)}\int_{(y_1W)\cap V}H^{-s}\mu\geq \lim_{s\to 1^+}(s-1)^{b(c)}\int_{(y_1)(W\cap V)}H^{-s}\mu>0,$$as claimed.
\end{proof}
\subsubsection{}In this paragraph we apply the Poisson formula for the height zeta function.
\begin{mydef}
For $s\in\CC$, by $Z(s)$ we denote the formal (non-empty) sum $$Z(s):=\frac{1}{\#G(F)}\sum_{x\in BG(F)}H(x)^{-s}.$$
\end{mydef}
The following fact follows from Proposition \ref{estimatenumber} and \cite[Chapter~2, Theorem 2.1]{widder}.
\begin{prop}\label{convz}
There exists $\eta>0$ such that the series defining~$Z$ converges absolutely and uniformly for $\Re(s)>\eta$. The function $s\mapsto Z(s)$ is holomorphic in the domain $\Re(s)>\eta$.
\end{prop}
The following conditions we need to verify in order to apply the Poisson formula holds.
\begin{lem}\label{condofpoiss}
Let $s\in\CC$ be such that $\Re(s)>\eta$. Let $A>0$ and let $g:BG(\AAA)\to\RR_{\geq 0}$ be a bounded continuous function. For every $x\in BG(\AAA)$, one has that the series $$Z^g(x,s):=\frac{\#\Sh^1(G)}{\# G(F)}\sum_{y\in i(BG(F))}g(xy)^{-s}H(xy)^{-s}$$ converges absolutely. The function $$Z^g(-,s):BG(\AAA)\to\CC,\hspace{1cm}x\mapsto Z^g(x,s)$$ is continuous.
\end{lem}
\begin{proof}
For $x\in \BGA$, we define $$\Sigma_x:=\Sigma\cup\{v\in M^0_F-\Sigma|i_v(x)\not\in B\GG(\Ov)\}$$ and $$U_{x}:=\prod_{v\in\Sigma_x}\{x_v\}\times\prod_{v\in M_F^0-\Sigma_x}B\GG(\Ov).$$ For every $x\in BG(\AAA)$, we are going to prove that the series $Z^g(x,s)$ converges absolutely and uniformly on $U_x$. For every $x'\in U_x$ and every $y\in i(BG(F))$, it follows from Lemma \ref{codx} that there exists $C(x)>0$ such that $$\big|H(x'y)^{-s}\big|=H(x'y)^{-\Re(s)}\leq C(x)^{-\Re(s)}H(y)^{-\Re(s)}.$$ Let $A>0$ be such that for every~$x\in BG(\AAA)$ one has $g(x)\leq A$. We have that 
\begin{align*}
\frac{\# \Sh^1(G)}{\# G(F)}\sum_{y\in i(BG(F))}\bigg|g(xy)^{s}H(xy)^{-s}\bigg|\hskip-3,7cm&\\&\leq \frac{\#\Sh^1(G)}{\# G(F)}\sum_{y\in i(BG(F))}A^{\Re(s)}C(x)^{-\Re(s)}H(y)^{-\Re(s)}\\
&=\frac{A^{\Re(s)}\cdot C(x)^{-\Re(s)}}{\# G(F)}{\sum_{y\in BG(F)}H(y)^{-\Re(s)}}\\
&=A^{\Re(s)}\cdot C(x)^{-\Re(s)}{Z(\Re(s))}.
\end{align*}
Now using Proposition \ref{convz}, we deduce that the series $Z^g(x,s)$ converges absolutely and uniformly in the domain $x'\in U_x$. It follows that $x'\mapsto Z^g(x',s)$ is continuous on $U_x$. We deduce that $x\mapsto Z^g(x,s)$ is continuous on $BG(\AAA)$, as claimed. 
\end{proof}
If $W$ is a locally compact abelian group endowed with a Haar measure~$dw$, we will denote by $dw^*$ the Haar measure on the character group~$W^*$ which is the dual of $dw$ \cite[Chapter II, \S 1, \no 3, Definition~4]{TSpectrale}. It is characterized by the following property: it is the unique Haar measure on~$W^*$ such that for every $f\in L^1(W, dw)$ for which the Fourier transform $\widehat f:\chi\mapsto\widehat f(\chi)=\int_{W}f\overline{\chi}dw$ satisfies $\widehat f \in L^1(W^*, dw^*)$, the \textit{Fourier inversion formula} \cite[Chapter II, \S 1, \no 4, Proposition 4]{TSpectrale} is valid: \begin{equation}\label{fourierinversion}f(w)=\int_{W^*}\overline{\chi(w)}\widehat{f}(\chi) dw^*(\chi)
\end{equation} for every $w\in W$.

In the next proposition, we apply the Poisson formula \cite[Chapter II, \S1, \no 8, Proposition 8]{TSpectrale}.
\begin{prop}\label{poisson}
In the domain $\Re(s)>\max(\eta,1)$,  one has that $$Z(s)=\frac{1}{\tau_{BG}}\sum_{\chi \in \Xi_H^{\perp}}\int_{BG(\AAA)}H^{-s}\overline{\chi}\mu.$$ 
\end{prop}
\begin{proof}
We will denote by $d\chi$ the dual Haar measure on the group $(BG(\AAA)/i(BG(F)))^*$ of the Haar measure $\mu/\coun_{i(BG(F))}$ on the group $BG(\AAA)/i(BG(F)).$ The space $(BG(\AAA)/i(BG(F))$ is compact and by Lemma \ref{tamagawabg}, one has that $\tau_{BG}= \frac{\# G^*(F)}{\# \Sh^2(G)},$ where 
\begin{align*}
\tau_{BG}=\bigg(\mu/\bigg(\frac{\# \Sh^1(G)\coun_{i(BG(F))}}{\# G(F)}\bigg)\bigg)\big(BG(\AAA)/(i(BG(F)))\big).
\end{align*}
Thus $$(\mu/\coun_{i(BG(F))})(BG(\AAA)/i(BG(F)))=\frac{\# G^*(F)\cdot \#\Sh^1(G)}{\# \Sh^2(G)\cdot \# G(F)}.$$
It follows that the dual measure on the discrete group $$i(BG^*(F))=i(BG(F))^{\perp}=(BG(\AAA)/i(BG(F)))^*$$ is given by$$d\chi=\frac{\# \Sh^2(G)\cdot \# G(F)}{\# G^*(F)\cdot \#\Sh^1(G)}\coun_{i(BG^*(F))}.$$
The necessary conditions of \cite[Chapter II, \S1, \no 8, Proposition 8]{TSpectrale} have been verified  in Lemma~\ref{condofpoiss} (we set $g=1$) and applying it gives:
\begin{align*}
Z(s)&=\frac{1}{\# G(F)}\sum_{x\in BG(F)}H(x)^{-s}\\
&=\frac{\#\Sh^1(G)}{\# G(F)}\sum_{x\in i(BG(F))}H(x)^{-s}\\
&=\frac{\# \Sh^1(G)\cdot\# \Sh^2(G)\cdot \# G(F)}{\# G(F)\cdot\# G^*(F)\cdot \# \Sh^1(G)}\sum_{\chi\in i(BG^*(F))}\wH(s,\chi)\\
&=\frac{1}{\tau_{BG}}\sum_{\chi\in (i(BG^*(F))}\wH(s,\chi),
\end{align*}
whenever every sum converges. It follows from Proposition \ref{convz} that the sums $\sum_{x\in BG(F)}H(x)^{-s}$ and $\sum_{x\in i(BG(F))}H(x)^{-s}$ converge in the domain $\Re(s)>~\eta$. By Corollary \ref{trivialcharacter}, for every $\chi\in BG(\AAA)^*$, the integral defining the Fourier transform $\wH(s,\chi)$ converges in the domain $\Re(s)>1$. Moreover, by Lemma \ref{importantchar}, one has for every $\Re(s)>1$ that $\wH(s,\chi)=0$ unless $\chi$ is an element of the finite set $(\Xi_H)^{\perp}.$ Hence, the sum $\sum_{\chi\in (i(BG^*(F))}\wH(s,\chi)$ reduces to the finite sum $\sum_{\chi\in\Xi_H^{\perp}}\wH(s,\chi).$ We deduce that $$Z(s)=\frac{1}{\tau_{BG}}\sum_{\chi\in \Xi_H^\perp}\wH(s,\chi)$$ for $\Re(s)>\max(1,\eta)$ as claimed.
\end{proof}
\begin{thm}\label{residueofz}
Let $\lambda>1$ be as in Lemma \ref{ftlfun}. The function $s\mapsto~Z(s)$ extends to a meromorphic function in the domain $\Re(s)>\lambda^{-1}.$ The meromorphic function $Z$ admits a pole of order $b(c)$ at $s=1$. One has that \begin{equation*}\lim_{s\to 1}(s-1)^{b(c)}Z(s)=\frac{\tau_H }{\tau_{BG}},\end{equation*}
where $$\tau_H:=\#\Xi_H^\perp\cdot\lim_{s\to 1^+}(s-1)^{b(c)}\int_{\#\Xi_H^\perp}H^{-s}\mu>0.$$
Moreover, if~$F$ is a number field, the function~$Z$ is holomorphic in the domain $\{\Re(s)\geq 1\}-\{1\}$ and if~$F$ is a function field, then the function~$Z$ is holomorphic in the domain $$\{\Re(s)\geq 1-\epsilon\}-\{1+2ki\pi\log(q)^{-1}|k\in\ZZ\}$$ for some $\epsilon>0$.
\end{thm}
\begin{proof}
For every $\chi\in \Xi_H^{\perp}$, by Part (1) of Corollary \ref{trivialcharacter}, one has that $s\mapsto\wH(s,\chi)$ extends to a meromorphic function in the domain $\Re(s)>\lambda^{-1}$ with a possible pole at $s=1$ of order at most $b(c)$. It follows from this fact and Proposition \ref{poisson} that the function $$s\mapsto Z(s)=\sum _{\chi\in \Xi_H^\perp}\wH(s,\chi)$$ extends to a meromorphic function in the domain $\Re(s)>\lambda^{-1}$ with a possible pole at $s=1$ of order at most $b(c)$.

Now, we are going to prove that the function $Z$ has a pole at $s=1$ which is of order $b(c)$. We will apply Poisson formula to the function $$BG(\AAA)^*\to\CC,\hspace{1cm}\chi\mapsto \frac{1}{\tau_{BG}}\wH(s,\chi)$$ when $\Re(s)>\max(\eta,1)$, for the closed inclusion of locally compact abelian groups $$\Xi_H^\perp\subset BG(\AAA)^*,$$ where the finite group $\Xi_H^\perp$ is endowed with the counting measure and the group $BG(\AAA)^*$ is endowed with the measure $\mu^*$. We first calculate the dual measure of the measure $(\mu/\coun_{\Xi_H^\perp})^*$ on $$(BG(\AAA)^*/\Xi_H^\perp)^*=\Xi_H.$$ By Lemma \ref{importantchar}, the group $\Xi_H$ identifies with the kernel of the surjective homomorphism $$BG(\AAA)=(BG(\AAA)^*)^*\to (\Xi_H^{\perp})^*,\hspace{1cm}x\mapsto x|_{\Xi_H^\perp}$$ and is an open and closed subgroup of the index $\# (\Xi_H^\perp)^*=\#\Xi_H^\perp$ in $BG(\AAA)$. The dual measure of the measure $\coun_{\Xi_H^\perp}$ on the dual group $(\Xi^\perp_H)^*$ is given by $\frac{1}{|\Xi_H^\perp|}\cdot~\coun_{(\Xi^\perp_H)^*},$ thus we have an equality of the Haar measures on $(BG(\AAA)^*/\Xi_H^\perp)^*=\Xi_H$: $$(\mu^*/\coun_{\Xi_H^\perp})^*=\#\Xi_H^\perp\cdot \mu|_{\Xi_H}.$$ 
Note that for $x\in (BG(\AAA)^*)^*=BG(\AAA),$ by the Fourier inversion formula (\ref{fourierinversion}), one has for $\Re(s)>\eta$ that
\begin{align*}
\int_{BG(\AAA)^*}\wH(s,\chi)\overline{x(\chi)}d\mu^*(\chi)\hskip-1cm&\\&=\int_{BG(\AAA)^*}\bigg(\int_{BG(\AAA)}H(y)^{-s}\overline{\chi(y)}d\mu(y) \bigg)\overline{x(\chi)}d\mu^*(\chi)\\
&= H(x)^{-s}.
\end{align*}
Now, for $\Re(s)>\max(\eta,1),$ the Poisson formula and Proposition \ref{poisson} give $$Z(s)=\sum_{\chi\in\Xi_H^\perp}\frac{1}{\tau_{BG}}\wH(s,\chi)=\frac{\#\Xi_H^\perp}{\tau_{BG}}\cdot\int_{\Xi_H}H^{-s}\mu.$$
 Corollary \ref{trivialcharacter} gives $$0<\lim_{s\to 1^+} (s-1)^{b(c)}\int_{\Xi_H}H^{-s}\mu.$$ 
and we deduce that the function $s\mapsto Z(s)$ admits a pole at $s=1$ of order exactly $b(c).$ The principal value at this pole is $$\lim_{s\to 1^+}(s-1)^{b(c)}Z(s)=\frac{ \#\Xi_H^\perp}{\tau_{BG}}\lim_{s\to 1^+}(s-1)^{b(c)}\int_{\Xi_H}H^{-s}\mu.$$ 
If~$F$ is a number field, then it follows from Part (1) of Corollary \ref{trivialcharacter} that $s\mapsto Z(s)=\sum_{\chi\in\Xi_H^{\perp}}\wH(s,\chi)$ is holomorphic in the domain $\{\Re(s)\geq 1\}-\{1\}$. If~$F$ is a function field, then it follows from Part (1) of Corollary \ref{trivialcharacter} that $s\mapsto Z(s)=\sum_{\chi\in\Xi_H^{\perp}}\wH(s,\chi)$ is holomorphic in the domain $\{\Re(s)\geq 1\}-\{1+2ki\pi\log(q)^{-1}|k\in\ZZ\}$.  The proof of the theorem is completed.
\end{proof}
\subsection{Equidistribution}\label{secequid}  We keep notations from previous subsection. Theorem \ref{residueofz} allows modifications of heights at finitely many places. By using \cite[Proposition 2.10]{Igusa}, this will give us that $G$-torsors are equidistributed in a subspace of the infinite product $\prod_{v\in M_F}BG(F_v)$ for a certain measure related to the height $H$.
 By abuse of notation, if $f:\prod_{\vMF}BG(F_v)\to \CC$ is a function, we may write also~$f$ for the composite map $f\circ i:BG(F)\to\CC$ and the composite map $BG(\AAA)\to \prod_{\vMF}BG(F_v)\xrightarrow{f}\CC$. We note that the inclusion $BG(\AAA)\to \prod_{\vMF}BG(F_v)$ is continuous. The functions of form~$\otimes_v f_v$ with~$f_v\equiv 1$ for almost all~$v$ will be called {\it elementary} functions.
\subsubsection{}In the following lemma we will define a measure on the product~$\prod_{\vMF}BG(F_v)$ for which rational points equidistribute. More practical formula will be provided in Proposition \ref{genfz}. 
\begin{lem}\label{existmeas}
For every continuous $f:\prod_{\vMF}BG(F_v)\to~\CC$, the series $$\frac{1}{\# G(F)}\sum_{x\in BG(F)}f(x)H(x)^{-s}$$converges absolutely to a holomorphic function in the domain $\Re(s)>1$ and the limit $$\omega_H(f):=\frac{\tau_{BG} }{\# G(F)}\lim_{s\to 1^+}(s-1)^{b(c)}\sum_{x\in BG(F)} f(x)H(x)^{-s}$$exists. Moreover, the map $f\mapsto\omega_H(f)$ defines a Radon measure on on the space $\prod_{\vMF}BG(F_v)$.
\end{lem}
\begin{proof}It follows from Theorem \ref{residueofz} and \cite[Section 2.3, Proposition~7]{coursdarithmetique} that the abscissa of absolute convergence of the Dirichlet series $Z(s)$ is $\Re(s)=1$.  Now, by boundedness of $f$, the first claim follows. Let us prove the existence of the limit.
\begin{enumerate}[label=(\alph*)]
\item We first suppose that $f=\otimes_v f_v$ is elementary, with $f_v$ strictly positive functions. 
One has that 
\begin{align*}
\frac{\tau_{BG} }{\# G(F)}\lim_{s\to 1^+}{(s-1)^{b(c)}}\sum_{x\in BG(F)} \big(f(x)^{-1}\cdot H(x)\big)^{-s}
\end{align*}exists by Theorem \ref{residueofz}. By the compactness of $\prod_{\vMF}BG(F_v)$, one has that $f^{s}\to f$ uniformly when $s\to 1^+$. Thus, one has that $$\frac{\tau_{BG} }{\# G(F)}\lim_{s\to 1^+}{(s-1)^{b(c)}}\sum_{x\in BG(F)} \big(f(x)^{s}-f(x)\big)\cdot H(x)^{-s} =0.$$By subtracting the later limit from the former, we deduce that $\omega_H(f)$ exists. 
\item Suppose that $f$ is a finite sum of functions $\otimes_v f_v$ as in (a). Then $\omega_H(f)$ exists. Moreover, if $\lambda\in\RR$ is scalar, then $\omega_H(\lambda\cdot(\otimes_v f_v))=\lambda\cdot \omega_H(\otimes_v f_v)$ exists.
\item Consider the family of functions which are of form $g(f_1\doots f_n)$, where $n\geq 1$ is an integer, $g\in\RR[X_1\doots X_n]$ is a polynomial and $f_i$ are as in (a). The functions from (a) are stable for multiplications, hence, this family is given by the linear combinations of elementary functions. 
\item The family of elementary functions contains constants and separates the points of $\prod_{\vMF}BG(F_v)$. By Stone-Weierstrass theorem \cite[Chapter X, \S 4, \no 2, Theorem 3]{TopologieGd} and by (c), one has that any continuous $f:\prod_{\vMF}BG(F_v)\to\RR$ is a uniform limit of functions which are linear combinations of functions of form $\otimes_v f_v$. Thus, the limit $\omega_H(f)$ exists. 
\item For a continuous complex valued function $f=\Re(f)+i\cdot\Im(f)$, one has that $\omega_H(f)$ exists because $\omega_H(\Re(f))$ and $\omega_H(\Im(f))$ exist. \end{enumerate}
Let us prove that~$\omega_H$ is a measure. If~$f$ is a non-negative function, then it is a limit of finite sums of non-negative functions~$\otimes_v f_v$, hence~$\omega_H(f)$ is non-negative. By \cite[Chapter III, \S 1, \no 6, Theorem 1]{Integrationj}, one has that~$\omega_H$ is a Radon measure. 
\end{proof}
\begin{lem}\label{supportofmeasureo}
One has that $$\supp(\omega_H)=\overline{i(BG(F))}$$ and that $$ \tau_H=\omega_H(\overline{i(BG(F))}).$$
\end{lem}
\begin{proof}
We denote by~$\mathbf 1_X$ the characteristic function of~$X$. If we denote by $U=\prod_{\vMF}BG(F_v)-i(\overline{BG(F)})$, then \begin{align*}\omega_H(U)&=\omega_H(\mathbf 1_U)\\&=\frac{\tau_{BG}}{\# G(F)}\lim_{s\to 1^+}{(s-1)^{b(c)}}\sum_{x\in BG(F)}\mathbf{1}_U(x)H(x)^{-s}\\&=0,
\end{align*}
thus we obtain that $$\supp(\omega_H)\subset\overline{i(BG(F))}.$$ 
Let us prove the reverse inclusion. Let $(y_v)_v\in i(BG(F)),$ let~$\Sigma_y\supset\Sigma_H$ be a finite set of places of~$F$ containing all places~$v$ for which $y_v\not\in BG(\Ov)$ and let $U_y\subset\prod_{\vMF}BG(F_v)$ be an open neighbourhood of~$(y_v)_v$ of the form $$\prod_{v\in\Sigma_y}\{y_v\}\times \prod_{v\in M_F-\Sigma_x}BG(F_v) .$$   We will prove that $\omega_H(U_y)>0$.  Let us write $$K=\prod_{v\in\Sigma_y}\{1\}\times\prod_{v\in M_F-\Sigma_y}BG(\Ov)$$ and $\Xi_K=i(BG(F))K\subset BG(\AAA)$, one has that the index~$(BG(\AAA):\Xi_K)$ is finite.  We denote by $j:BG(\AAA)\to \prod_{\vMF}BG(F_v)$ the inclusion, which is a continuous homomorphism. One has that 
\begin{align*}j^{-1}(U_y)&=j^{-1}\bigg(\prod_{v\in\Sigma_y}\{y_v\}\times \prod_{M_F-\Sigma_y}BG(F_v)\bigg)\\&=\yyy j^{-1}\bigg(\prod_{v\in\Sigma_y}\{1\}\times\prod_{v\in M_F-\Sigma_y}BG(F_v)\bigg)\\
&=\yyy V,
\end{align*}
where $\yyy=((y_v)_{v\in \Sigma_y},(1)_{v\in M_F-\Sigma_y})$ and $$V=j^{-1}\bigg(\prod_{v\in\Sigma_y}\{1\}\times\prod_{v\in M_F-\Sigma_y}BG(F_v)\bigg).$$ Note that $\yyy\in\Xi_K$ and that the subgroup $V\subset BG(\AAA)$ is of finite index. 
Now, Lemma \ref{intfinind}, the fact that $\sum_{\chi\in\Xi_K^{\perp}}\chi=\#\Xi_K^{\perp}\cdot\mathbf{1}_{\Xi_K}$ and Proposition~\ref{thmofequidistribution} give that
\begin{align*}
\omega_H(U_y)&=\omega_H(\mathbf{1}_{U_y})\\
&=\tau_{BG}\lim_{s\to 1^+}(s-1)^{b(c)}Z^{\mathbf{1}_{U_y}}(s)\\
&=\lim_{s\to 1^+}(s-1)^{b(c)}\sum_{\chi\in\Xi_K^{\perp}}\int_{BG(\AAA)}(\mathbf 1_{j^{-1}(U_y)}\cdot H^{-s})\chi\mu\\
&=\lim_{s\to 1^+}(s-1)^{b(c)}\#\Xi_K^{\perp}\cdot \int_{BG(\AAA)}\mathbf 1_{\Xi_K}\mathbf 1_{j^{-1}(U_y)}H^{-s}\mu\\
&=\#\Xi_K^{\perp}\cdot\lim_{s\to 1^+}(s-1)^{b(c)}\int_{\Xi_K\cap j^{-1}(U_y)}H^{-s}\mu\\
&>0.
\end{align*}
It follows that $\supp(\omega_H)\supset i(BG(F))$ and hence that $$\supp(\omega_H)=\overline{i(BG(F))}.$$
Finally, we have that \begin{align*}\omega_H(\overline{i(BG(F))})&=\omega_H\big(\prod_{\vMF}BG(F_v)\big)\\&=\omega_H(\mathbf{1}_{\prod_v BG(F_v)})\\&= \tau_{BG}\lim_{s\to 1^+}\frac{(s-1)^{b(c)}}{\# G(F)}\sum_{x\in BG(F)}H(x)^{-s}\\&=\tau_{BG}\cdot\lim_{s\to 1^+}(s-1)\cdot Z(s)\\&=\tau_H.
\end{align*} 
\end{proof}
\begin{rem}
\normalfont Let~$F$ be a number field. Harari proves in \cite[Page 522]{haraprox} that {\it Brauer-Manin obstruction is the only obstruction to the weak approximation for~$G$}. Let us explain what he means by this. We assume that~$G$ is embedded as a closed subgroup scheme of $\GL_n$ for certain $n\geq 1$ and let us set $X=\GL_n/G$ to be the quotient variety. Then the closure of the image of $X(F)$ for the diagonal map to $\prod_{\vMF}X(F_v)$ is the Brauer-Manin set $\big(\prod_{\vMF}X(F_v)\big)^{\Br}$, i.e. the subset of $\prod_{\vMF}X(F_v)$ where the Brauer-Manin obstruction vanishes. Moreover, he explains that one could have chosen other embedding in $\GL_n$ or in $\SL_n$ and the same property would be true. Demeio told us that one can define Brauer-Manin set $\big(\prod_{\vMF}BG(F_v)\big)^{\Br}$ as the image of $\big(\prod_{\vMF}X(F_v)\big)^{\Br}$ for the induced map $\prod_{\vMF}X(\Fv)\to\prod_{\vMF}BG(F_v)$ (although there is a better definition in terms of certain pairings).  Then it is not hard to verify that the fact that Brauer-Manin obstruction is the only obstruction to the weak approximation for~$G$ translates into  $$\overline{i(BG(F))}=\bigg(\prod_{\vMF}BG(F_v)\bigg)^{\Br}.$$ 
We may compare this with Colliot-Th\' el\`ene's conjecture \cite[Conjecture 14.1.2]{brauergrothendieck} that the Brauer-Manin obstruction is the only obstruction to the weak approximation for rationally connected varieties.
\end{rem}
\subsubsection{}We modify the Dirichlet series by multiplying by the value of~$f$. We calculate the limit when the new series is multiplied by $(s-1)^{b(c)}$ and $s\to 1$.
\begin{prop}\label{genfz}
Let $f:\prod_{\vMF}BG(F_v)\to \RR_{\geq 0}$ be a continuous function. We set $$Z^f(s):=\frac{1}{\# G(F)}\sum_{x\in BG(F)}f(x)^sH(x)^{-s}.$$ 
\begin{enumerate}
\item For every~$\epsilon>0$, the series $Z^f(s)$ converges absolutely and uniformly to a holomorphic function in the domain $\Re(s)>1+\epsilon$. One has that $$\lim_{s\to 1^+}(s-1)^{b(c)}Z^f(s)=\frac{\omega_H(f)}{\tau_{BG}}.$$
\item Suppose that $f=\sum_{i=1}^n\lambda_if_i$, with $\lambda_i\in\RR$ and~$f_i$ elementary. Let $\Sigma$ be the set of all places containing $\Sigma_G\cup M_F^\infty$ such that for every $v\in M_F-\Sigma$ one has that $(f_i)_v\equiv 1$ and that~$H_v$ is~$BG(\Ov)$-invariant. One has that \begin{align*}
Z^f(s)&=\frac{1}{\tau_{BG}}\sum_{\chi\in \Xi^{\perp}_K}\int_{BG(\AAA)}f^sH^{-s}\chi\mu\\&=\frac{1}{\tau_{BG}}\sum_{\chi\in \Xi_K^{\perp}}\bigg(\int_{\prod_{v\in \Sigma}BG(F_v)}f_{\Sigma}^sH_{\Sigma}^{-s}\chi_{\Sigma}\mu_{\Sigma}\bigg) \cdot \prod_{v\in M_F-\Sigma}\wH_v(s,\chi_v),\end{align*}
where~$K=\prod_{\vMFz-\Sigma}BG(\Ov),$ and $\Xi_K^{\perp}$ denotes the set of characters in $BG^*(F)$ which vanish on~$K$, while the subscript~$\Sigma$ denotes the restriction of a function or a measure to $\prod_{v\in\Sigma}BG(F_v)$. If~$f\neq 0$, the function $s\mapsto Z^f(s)$ extends to a meromorphic function in the domain $\Re(s)\geq 1$ and admits a pole at $s=1$ of order~$b(c)$. Moreover, if~$F$ is a number field, then $s\mapsto Z^f(s)$ is holomorphic in the domain $\{\Re(s)\geq 1\}-\{1\}$ and if~$F$ is a function field, then $s\mapsto Z^f(s)$ is holomorphic in the domain $\{\Re(s)\geq 1-\epsilon\}-\{1+2ki\pi\log(q)^{-1}|k\in\ZZ\}$ for some~$\epsilon>0$.
\end{enumerate}
\end{prop}
\begin{proof}
\begin{enumerate}
\item   It follows from Theorem \ref{residueofz} and \cite[Section 2.3, Proposition~7]{coursdarithmetique} that the Dirichlet series $Z(s)$ converges absolutely and uniformly in the domain $\Re(s)\geq 1+\epsilon$. The claim on the convergence and holomorphicity of $Z^f(s)$ in the domain $\Re(s)>1+\epsilon$ now follows by the fact that~$f$ is bounded. By the same theorem, one has $\lim_{s\to 1^+}(s-1)^{b(c)}Z(s)=\tau_H$ and by the compactness of $\prod_{\vMF}BG(F_v)$, one has that $f^s\to f$ uniformly. We deduce that $$\lim_{s\to 1^+}\frac{(s-1)^{b(c)}}{\# G(F)}\sum_{x\in BG(F)}(f(x)^s-f(x))H(x)^{-s}=0.$$ 
Thus
\begin{align*}\lim_{s\to 1^+}(s-1)^{b(c)}Z^f(s)&=\lim_{s\to 1^+}\frac{(s-1)^{b(c)}}{\# G(F)}\sum_{x\in BG(F)}f(x)H(x)^{-s}\\&\quad+\lim_{s\to 1^+}\frac{(s-1)^{b(c)}}{\# G(F)}\sum_{x\in BG(F)}(f(x)^s-f(x))H(x)^{-s}\\
&=\frac{\omega_H(f)}{\tau_{BG}}.
\end{align*}
\item 
In Lemma~\ref{condofpoiss}, we have verified the conditions of Poisson formula, and applying it gives:
\begin{align*}
Z^f(s)\hskip-2cm&\\&=\frac{1}{\tau_{BG}}\sum_{\chi\in \Xi^{\perp}_K}\int_{BG(\AAA)}f^sH^{-s}\chi\mu\\
&=\frac{1}{\tau_{BG}}\sum_{\chi\in \Xi^{\perp}_K}\bigg(\int_{\prod_{v\in \Sigma'}BG(F_v)}f_{\Sigma'}^sH_{\Sigma'}^{-s}\chi_{\Sigma'}\mu_{\Sigma'}\bigg) \cdot \prod_{v\in M_F-\Sigma'}\wH_v(s,\chi_v),
\end{align*}
whenever every quantity converges. It is evident that $$s\mapsto \int_{\prod_{v\in \Sigma'}BG(F_v)}f_{\Sigma'}^sH_{\Sigma'}^{-s}\chi_{\Sigma'}\mu_{\Sigma'}$$ is an entire function. By Corollary \ref{trivialcharacter}, one has that $\prod_{v\in M_F-\Sigma'}\wH_v(s,\chi_v)$ converges absolutely to a holomorphic function in the domain $\Re(s)>1$ and extends to a meromorphic function in the domain $\Re(s)\geq 1$ with a possible pole at $s=1$ of order no more than~$b(c)$.  We deduce that the same is valid for $s\mapsto Z^f(s)$. As $f\neq 0$, it follows from (1) that the pole exists and that is of order $b(c)$. Moreover, if $F$ is a number field, the function $s\mapsto\prod_{v\in M_F-\Sigma'}\wH_v(s,\chi_v)$ is holomorphic in the domain $\{\Re(s)\geq 1\}-\{1\}$. If~$F$ is a function field, the function $s\mapsto\prod_{v\in M_F-\Sigma'}\wH_v(s,\chi_v)$ is holomorphic in the domain $\{\Re(s)\geq 1-\epsilon\}-\{1+2ki\pi\log(q)^{-1}|k\in\ZZ\}$ for some~$\epsilon>0$. The two claims on holomorphicity of~$Z^f$ follow. The statement is proven.
\end{enumerate}
\end{proof}
\begin{rem}
\normalfont
By combining (1) and (2) of Proposition \ref{genfz}, we obtain a formula which can be used to compute $\omega_H(f)$ when $f=\sum \lambda_if_i,$ with~$\lambda_i\in\RR$ and~$f_i$ are elementary. By linearity, it is possible to write down the formula for~$\lambda_i\in\CC$. 
\end{rem}
\begin{rem}
\normalfont We remark that, contrary to the case of varieties,  the measure~$\omega_H$ is not always an Euler product of local measures, but only a finite sum of Euler products. This may happen even when $\overline{i(BG(F))}=\prod_{\vMF}BG(F_v)$ (e.g. when $G=\muu_m$). 
\end{rem}
\begin{thm}\label{satfunf}
Suppose that~$F$ is a function field. Let $$f:\prod_{\vMF}BG(F_v)\to \RR_{\geq 0}$$ be a continuous function such that~$\omega_H(f)> 0.$ 
One has that $$\#\{x\in BG(F)|H(x)\leq f(x)B\}\asymp_{B\to\infty} B\log(B)^{b(c)-1}.$$
\end{thm}
\begin{proof}
We split the proof in several parts. 
We start by making several simplifications on~$H$. 
\begin{enumerate}
\item We can suppose that~$c$ takes rational values. Indeed, for $r\in c(G_*(\oF))-\{0,1\}$, let us choose $r', r''\in\QQ_{>0}$ such that $1<r'<r<r''$ and such that for every $r, t\in c(G_*(\oF))-\{0,1\}$ one has that if $r<t$ then $r'<t'$ and $r''<t''$. We define a counting function $c':G_*(\oF)\to\QQ_{\geq 0}$ by $c'(x):= c(x)$ if $c(x)\in\{0,1\}$, otherwise $c'(x)=r'$, where $r=c(x)$. Analogously, one defines~$c''$. One can define a height~$H'$ by~$H'_v=H_v$ for $v\in\Sigma_{H}$ and $H'_v=q_v^{c'(\psi_v(x))}$. Analogously, one defines~$H''$. We have that that $H'(x)\leq H(x)\leq H''(x)$ for every $x\in BG(F)$. Thus if the statement is valid for the heights~$H'$ and~$H''$, it is also valid for the height~$H$. 
\item We prove that we can suppose that $f=\mathbf 1_U$ is the characteristic function of an open set~$U$ and that it is elementary (we call such open subsets elementary). The elementary open sets form a basis for the topology of $\prod_{\vMF}BG(F_v)$, hence the open set $\{x\in\prod_{\vMF}BG(F)|f(x)>0\}$ contains an elementary open~$V$. By finiteness of the spaces $BG(F_v)$, the elementary open sets are also closed. Hence, the set~$V$ is compact and it follows that there exists $\epsilon>0$ such that for every $x\in\prod_{\vMF}BG(F_v)$ one has that $f(x)\geq \epsilon\cdot\mathbf 1_V$. On the other side the function~$f$ is bounded and let~$C>0$ be such that for every $x\in\prod_{\vMF}BG(F_v)$, one has that $f(x)\leq C\cdot\mathbf 1_{\prod_{\vMF} BG(F_v)}$. 
\begin{align*}
\{x\in BG(F)|H(x)\leq \epsilon\mathbf 1_{V}B\}&\subset\{x\in BG(F)| H(x)\leq f(x)B\}\\&\subset \{x\in BG(F)|  H(x)\leq  CB\}.
\end{align*}
It follows thus to prove the statement for the characteristic function~$f=\mathbf 1_U$ of elementary open sets~$U$.
\item 
Changing~$H$ at finitely many places will not change the validity of the statement, hence we can suppose that there exists a finite set of places $\Sigma_G\subset\Sigma$ such that for every $v\in\Sigma$ one has that~$H_v=1$. 
This simplification provides that 
$$H(BG(F))\subset \{q^{\frac mr}|m\in\ZZ_{\geq 0}\}$$ for some~$r\in\ZZ_{>0}$.  
\item Let us complete the proof. For $n\in\ZZ_{\geq 0}$, we denote by~$a_n$ the number of $x\in BG(F)$ with $i(x)\in U$ having the height equal to~$q^{n/r}$, counted with multiplicity $\# G(F)^{-1}$. Let us introduce the variable $t=~q^{-s/r}$. We define power series $Q(t)=\sum _{n=0}^{\infty}a_nt^{n}$ so that $$Q(t)=Q(q^{-s/r})=Z^f(s).$$
By Proposition \ref{genfz}, the series defining $Z^f(s)$ converges absolutely for $\Re(s)> 1$ to a holomorphic function, the function $s\mapsto Z^f(s)$ extends to a meromorphic function in the domain $\Re(s)>1-\epsilon$ for some~$\epsilon>0.$ Moreover, as $\omega_H(\mathbf 1_U)=\omega_H(f)>0$, the function~$Z^f$ has a pole at~$s=1$ of order~$b(c)$ and no poles outside of the set $\{1+2ik\pi\log(q)^{-1}|k\in\ZZ\}$. 
Hence, the series~$Q(t)$ converges absolutely for $|t|<q^{-1/r}$ to a holomorphic function and it extends to a meromorphic function of the disk $|t|<q^{-1/r}+\epsilon$ for some~$\epsilon$ with the only pole at~$t=q^{-1/r}$ of order $b(c)$. By Part~(c) of \cite[Theorem A.5]{lagemanartinschreier}, for some $C>0$, one has that $$\sum_{n=0}^Ma_n\sim_{M\to\infty} C\cdot q^{M/r}\log(q^{M/r})^{b(c)-1}.$$
Finally, $B\asymp_{B\to\infty} q^{\lfloor\log_{q}(rB)\rfloor/r}$, and hence 
\begin{align*}
\#\{x\in BG(F)|i(x)\in U, H(x)\leq B\}\hskip-2cm&\\&=\sum_{n=0}^Ba_n\\&=\sum_{n=0}^{\lfloor\log_{q}(rB)\rfloor}a_n\\
&\sim_{B\to\infty}C\cdot q^{\lfloor \log_{q}(rB)\rfloor/r}\log(q^{\lfloor\log_{q}(rB) \rfloor/r})^{b(c)-1}\\
&\asymp_{B\to\infty} B\log(B)^{b(c)-1}.
\end{align*}
The statement follows.
\end{enumerate}
\end{proof}
When~$F$ is a number field, we prove that $G$-torsors are equidistributed in the product space $\prod_{\vMF}BG(F_v)$.
\begin{thm}
\label{thmofequidistribution} Suppose $F$ is a number field. We define a Borel measure~$\nu$ on $\prod_{\vMF}BG(F_v)$ by $$\nu:=\frac{i_*\coun_{BG(F)}}{\#G(F)}.$$
 For $B>0$, we define a Borel measure~$\nu_B$ on $\prod_{\vMF}BG(F_v)$ by $$\nu_B:=\frac{\nu}{\frac{\#\Sh^2(G)}{\# G^*(F)\cdot (b(c)-1)!}\cdot B\cdot \log(B)^{b(c)-1}}\cdot\mathbf 1_{\{(x_v)_v|H((x_v)_v)\leq B\}}.$$ The measures~$\nu_B$ converge vaguely to the measure~$\omega_H$ when~$B\to\infty$. In other words,
\begin{enumerate}
\item for any continuous $f:\prod_{\vMF}BG(F_v)\to\CC$ one has that $$\lim_{B\to\infty}\frac{\frac{1}{\# G(F)}\sum_{\substack{x\in BG(F)\\H(x)\leq B}}f(x)}{B\cdot\log(B)^{b(c)-1}}=\frac{\#\Sh^2(G)}{\# G^*(F)\cdot (b(c)-1)!}\cdot \omega_H(f),$$
\item for every open set $U\subset\prod_{\vMF}BG(F_v)$ the boundary of which is $\omega_H$-negligible, one has that \begin{multline*}\frac{\#\{x|i(x)\in U, H(x)\leq B\}}{\# G(F)}\\=\frac{\#\Sh^2(G)\cdot\omega_H(U)}{\# G^*(F)\cdot (b(c)-1)!}\cdot {B\log(B)^{b(c)-1}}+o(B\log(B)^{b(c)-1}).
\end{multline*}
\end{enumerate} 
 %
%
%
\end{thm}
\begin{proof}
Let~$f:\prod_{\vMF}BG(F_v)\to\RR_{\geq 0}$ be a function which is a finite linear combination of elementary functions. By Proposition~\ref{genfz} and the Tauberian result from \cite[Theorem 3.3.2]{aniso}, one has that \begin{align*}\nu(\{x\in i(BG(F))| H(x)\leq f(x)\cdot B\})\sim_{B\to\infty}\frac{\#\Sh^2(G)\cdot\omega_H(f)}{\#G^*(F)\cdot(b(c)-1)!}{B\log(B)^{b(c)-1}}.
\end{align*} 
By Part (c) and Part (d) of proof of Lemma \ref{existmeas}, we have that the functions, which are finite linear combinations of elementary functions~$\otimes_vf_v$, are dense in the space of continuous real valued functions $\prod_{\vMF}BG(F_v)\to~\RR$. The statement is now a special case of \cite[Proposition 2.10]{Igusa}. 
\end{proof}
We have noted in Paragraph~\ref{parbreakthin} that in the commutative case there are no breaking thin maps. Thus Theorems~\ref{satfunf} and~\ref{thmofequidistribution}, together with the fact that for every height~$H$ of type~$c$, the type of the height~$H^{\frac{1}{a(c)}}$ is normalized, imply that Conjecture~\ref{malletale} is verified for~$G$ commutative. Moreover, in the number field case, for non-normalized heights~$H$ of type~$c$ we get that 
\begin{align*}\frac{\#\{x\in BG(F)|H(x)\leq B\}}{\# G(F)}=\frac{\#\Sh^2(G)\cdot\omega_H\big(\overline{i(BG(F))}\big)}{\# G^*(F)\cdot (b(c)-1)!}B^{a(c)}\log(B)^{b(c)-1}.
\end{align*}
\begin{rem}
 \normalfont Peyre conjectures in \cite{Peyre}, that for a smooth Fano $F$-variety~$X$, the leading constant in the asymptotic formula for the number of points of bounded height is a product of three factors: a factor connected to the location of the anticanonical line bundle in the ample cone, the cardinality of the cohomology group $\# H^1(\Gamma_F, \Pic(\overline F))$ and the volume of the subset~$X(\overline F)\subset \prod_{\vMF}X(F_v)$ for a measure analogue to~$\omega_H$. In our situation, the volume factor appears, but we are unsure whether the other factors can be understood as in the case of varieties.
\end{rem}
\bibliography{bibliography}
\bibliographystyle{acm}
\end{document}